\documentclass{amsart}

\usepackage[a4paper, outer=3.4cm, inner=3.4cm, top=2cm, bottom=2cm, marginparsep=.2cm, centering]{geometry}
\usepackage[T1]{fontenc}
\usepackage{comment}
\usepackage{microtype}
\usepackage[textsize=scriptsize]{todonotes}
\usepackage{silence}
\newcommand{\Spa}{\operatorname{Spa}}
\newcommand{\Spf}{\operatorname{Spf}}

\newcommand{\Spec}{\operatorname{Spec}}

\newcommand{\intcl}[1]{\overline{#1}^{\raisebox{-0.3ex}{\hspace{0.2mm}}}}

\newcommand{\overhat}[1]{\mkern 1.3mu\widehat{\mkern-1.3mu#1\mkern-1.3mu}\mkern 1.3mu}

\newcommand{\sptag}[1]{\href{https://stacks.math.columbia.edu/tag/#1}{\texttt{#1}}}
\NewDocumentCommand{\spcite}{>{\SplitArgument{2}{,}}m}{\spciteaux#1}
\newcommand{\spciteaux}[3]{\IfNoValueTF{#2}{\cite[Tag~\sptag{#1}]{StacksProject}}
	{\IfNoValueTF{#3}{\cite[Tags~\sptag{#1}, \sptag{#2}]{StacksProject}}
	{\cite[Tags~\sptag{#1}, \sptag{#2}, \sptag{#3}]{StacksProject}}}}
\newcommand{\restrsmall}[2]{{%
  #1
  \vphantom{|}
  |_{#2}
}}
\newcommand*{\isomarrow}{\xrightarrow{\raisebox{-0.35em}{\smash{\ensuremath{\sim}}}}}
\newcommand*{\isomlongarrow}{\xlongrightarrow{\raisebox{-0.35em}{\smash{\ensuremath{\sim}}}}}

\usepackage{tikz-cd}

\usepackage{quiver}
\usetikzlibrary{nfold}
\usetikzlibrary{babel}         % Important: For TikZ and Babel compatibility
\usepackage[shortlabels]{enumitem}
\usepackage{graphicx}
\usepackage{amsmath, amssymb, amsthm, thmtools, mathtools, amsfonts, stmaryrd, mathtools}
\usepackage{mathrsfs}     
\newcommand{\genf}{\vphantom{X}_\eta^{\mathrm{ad}}}

\renewcommand{\sp}{\ensuremath{\mathrm{sp}}}

\newcommand{\SKKc}{{\operatorname{Spa}(K, K^\circ)}}
\newcommand{\SKKp}{{\operatorname{Spa}(K, K^+)}}
\newcommand{\SKc}{{\operatorname{Spf}(K^\circ)}}
\newcommand{\SKp}{{\operatorname{Spf}(K^+)}}
\newcommand{\Frac}{\operatorname{Frac}}
\newcommand{\Mor}{\operatorname{Mor}}
\newcommand{\Hom}{\operatorname{Hom}}

\newcommand{\cA}{\mathcal{A}}
\newcommand{\cB}{\mathcal{B}}

\newcommand{\cO}{\mathcal{O}}
\newcommand{\cE}{\mathcal{E}}
\newcommand{\cF}{\mathcal{F}}
\newcommand{\cN}{\mathcal{N}}
\newcommand{\cT}{\mathcal{T}}
\newcommand{\cG}{\mathcal{G}}
\newcommand{\cH}{\mathcal{H}}

\newcommand{\cS}{\mathcal{S}}
\newcommand{\cU}{\mathcal{U}}

\newcommand{\cX}{\mathcal{X}}
\newcommand{\cY}{\mathcal{Y}}

\newcommand{\dA}{\mathbb{A}}
\newcommand{\dB}{\mathbb{B}}
\newcommand{\dC}{\mathbb{C}}

\newcommand{\dG}{\mathbb{G}}

\newcommand{\dQ}{\mathbb{Q}}

\newcommand{\dZ}{\mathbb{Z}}

\usepackage{hyperref}
\hypersetup{
    pdftitle={Abeloid varieties over higher rank valued fields: Raynaud uniformization and N\'eron models},
    pdfauthor={Simon Kaib, Annette Werner},
    colorlinks, citecolor=blue, filecolor=black, linkcolor=black, urlcolor=black, hypertexnames=false}
\usepackage[capitalize]{cleveref}

\newtheorem{theorem}{Theorem}[section]
\newtheorem{lemma}[theorem]{Lemma}
\newtheorem{proposition}[theorem]{Proposition}
\newtheorem{corollary}[theorem]{Corollary}

\newtheorem{convention}[theorem]{Convention}

\newtheorem{introtheorem}{Theorem}

\newtheorem{introcorollary}[introtheorem]{Corollary}
\crefname{introtheorem}{Theorem}{Theorems}
\Crefname{introtheorem}{Theorem}{Theorems}
\crefname{introcorollary}{Corollary}{Corollaries}
\Crefname{introcorollary}{Corollary}{Corollaries}

\newtheorem{definition}[theorem]{Definition}
\newtheorem{example}[theorem]{Example}

\newtheorem{remark}[theorem]{Remark}

\title[Abeloid varieties over higher rank valued fields]{Abeloid varieties over higher rank valued fields: Raynaud uniformization and N\'eron models}

\author{Simon Kaib, Annette Werner}
\date{September 2026}

\begin{document}

\begin{abstract}
We study abeloid varieties over $\SKKp$, where $K$ is a complete non-Archimedean field and $K^+\subseteq K^\circ$ is a valuation ring of arbitrary rank. We extend the well-known Raynaud uniformization to this higher-rank setting: after a finite separable base change, an abeloid variety $A$ over $\SKKp$ admits a uniformization as $A\cong E/M$, where $E$ is an extension of an abeloid variety $B$ with generically good reduction by a torus, and $M$ is a full lattice. We also  show the existence of suitable N\'eron models in algebraic and formal situations over higher rank valuation rings: We prove the existence of algebraic (lft-)Néron models for semi-abelian varieties over valuation rings with algebraically closed fraction fields, as well as, up to finite separable base change, over complete rank one valuation rings with algebraically closed residue field. Building on this result, we show the existence of marked Néron models for abeloid varieties over $\SKKp$ after finite separable base change.
We identity
the failure of generically good reduction to extend to good reduction of $B$ over $K^+$ as the obstruction to recovering the higher rank Raynaud extension directly from the formal {N\'eron} model.
Finally, we prove that multiplication by $m$ is finite and faithfully flat on Néron models {in the algebraically closed setting}, and we illustrate our results for elliptic curves.
\end{abstract}

\maketitle

\centerline{MSC(2020) 14K15, 14G22}

\section{Introduction}

The present paper investigates abeloid varieties $A$ over $\SKKp$, where $K$ is  a complete non-Archimedean field with non-trivial valuation and $K^+$ is a valuation ring of $K$ of arbitrary rank.
We prove that in this setting the abeloid variety $A$ admits a uniformization by a semi-abelian 
{space} which generalizes the well-known Raynaud uniformization in the rank one case.
We also establish existence results for algebraic and formal Néron models over higher rank valuation rings.
The algebraic existence results extend the known theory even in the classical case of {non-discrete} rank one valuation rings.

Our results generalize the following picture over rank one valuation rings. Let $K^\circ=\cO_K$ be the rank one valuation ring defining $K$.  After passing to a finite separable extension of $K$, an abeloid variety
$A$ over $\Spa(K,K^\circ)$ admits a (Raynaud) uniformization as 
  $A \cong E/M$, 
where $E$ is an extension 
of an abeloid variety $B$ with good reduction by a split torus 
$T\cong (\dG_m^\mathrm{an})^r$, 
and
$M \cong \underline{\dZ^r}$ is a full lattice in $E$. In the algebraic case this is shown in \cite{BoschLuetkebohmert1991Degenerating}, in the general case of abeloid varieties in \cite{Luetkebohmert2016RigidGeometryCurves}, see \cref{t:classical_uniformization_analytic}.
This analytic uniformization is closely related to the reduction theory of $A$,  see \cref{t:classical_uniformization_formal1}.

The purpose of this paper is to study this picture when the rank one
valuation ring $K^\circ$ is replaced by a possibly higher rank
valuation ring $K^+ \subseteq K^\circ$ of $K$.
{In the language of Huber's adic spaces, which we use in the present paper,  this amounts to replacing the rank one base $\operatorname{Spa}(K,K^\circ)$ by $\operatorname{Spa}(K,K^+)$. Such bases arise naturally as completed residue fields of analytic adic spaces, and hence in the study of fibers of morphisms of adic spaces.

For every adic space $X$ over $\SKKp$, we write
\begin{equation*}
    X^\circ := X\times_{\operatorname{Spa}(K,K^+)} \operatorname{Spa}(K,K^\circ)\vspace{-0.7mm}
\end{equation*}
for its rank one fiber.} 
If $X$ is partially proper we show in \cref{p:valuative_strong_criterion} that passing to the rank one fiber loses little information. In fact we have
\begin{equation*}
    ((-)^\circ)_*X^\circ=X.
\end{equation*}
By compactifying formal models, we prove that $(-)^\circ$ induces an equivalence of categories between proper adic spaces of topologically finite type over $\SKKp$ and $\SKKc$, respectively; see \cref{p:equivalence_proper}.
This induces an equivalence between abeloid varieties over the two bases, and allows us to lift the classical uniformization cross of $A^\circ$ to the higher rank base.

We say that an abeloid variety over $\SKKp$ has \textit{generically good reduction} if its $\SKKc$-fiber has good reduction. One of our main results, \cref{t:uniformization_analytic},
shows that Raynaud uniformizations exist over $\SKKp$:

\begin{introtheorem}\label{t:intro1}
    Let $A$ be an abeloid variety over $\SKKp$.
    After a finite separable affinoid field extension,
    there exists an isomorphism (as sheaves on the analytic site)
    \begin{equation*}
        A\cong E/M,
    \end{equation*}
    where
    \[ 0 \longrightarrow T \longrightarrow E \longrightarrow B \longrightarrow 0\]
    is an extension of an abeloid variety $B$ with \textit{generically good reduction}
    by a split algebraic torus $T\cong \dG_{m,\SKKp}^r$, and $\underline{\dZ^r}\cong M\subseteq{} E$ is a full lattice. 
   
    The resulting uniformization cross is uniquely determined by $A$ and functorial in $A$. Its $\SKKc$-fiber is the classical uniformization cross of $A^\circ$, and $E\to A$ is the universal analytic covering space of $A$.
\end{introtheorem}

There
is an essential new geometric feature in higher rank: generically good
reduction need not imply good reduction over $K^+$. Thus, although
the good reduction part of the classical Raynaud extension of
$A^\circ$ always extends as an abeloid variety over
$\operatorname{Spa}(K,K^+)$, it need not admit a
smooth formal group model 
over $\operatorname{Spf}(K^+)$. This phenomenon already occurs for
elliptic curves.  If $E$ is an elliptic curve over
$\operatorname{Spa}(K,K^+)$ with
\[
    j(E)\in K^\circ\setminus K^+,
\]
then, up to a finite separable field extension, $E$ has generically good reduction, but it does not have good
reduction over $K^+$; see
\cref{e:classification_RU_curves}.  

We also investigate the natural question whether we have nice formal models over $K^+$ of abeloid varieties over $\SKKp$. 
In the case $K^+ = K^\circ$ of rank one,  L\"utkebohmert {(see \cite[Section 7.6]{Luetkebohmert2016RigidGeometryCurves})} constructs a so-called marked N\'eron model for an abeloid variety  $A$ over $\SKKc$ after a finite separable base change of the ground field. Note that this is in general not a model of $A$, but only of an open subgroup. For the definition of marked N\'eron models see \cref{d:markedNeron_K+}.

In order to construct marked N\'eron models over higher rank $K^+$, we first
prove the following new existence result for algebraic (lft-)N\'eron models
in \cref{t:Neron_existence_V}
that may be of
independent interest:

\begin{introtheorem}\label{t:intro2}
Let $V$ be a Henselian valuation ring with fraction field $F$, and suppose that the residue fields of $\Spec(V)$ at non-generic points are algebraically closed. Let $E_F\to \Spec(F)$ be a semi-abelian variety. After replacing $F$ by a finite separable field extension and $V$ by the unique valuation ring dominating it, $E_F$ extends to a semi-abelian scheme over $\Spec(V)$.
Once it does, $E_F$ admits a (lft-)Néron model over $\Spec(V)$ whose fiberwise identity component is a semi-abelian scheme.
\end{introtheorem}

Note that the assumptions on $V$ are satisfied in the case that $F$ is algebraically closed, as well as in the case that $(F,V)=(K,K^+)$ has
algebraically closed residue field $k=K^\circ/{K^{\circ\circ}}$, see \cref{e:Neron_excistence}. This applies in particular to the rank one case $K^+=K^\circ$. For example it shows that every abelian variety over $\dC_p$ admits a (lft)-Néron model over $\cO_{\dC_p}$, which seems to be missing from the literature so far.

A key step is the observation that a
semi-abelian scheme over a normal integral base already satisfies an algebraic version of the
marked N\'eron mapping property; see
\cref{p:semi-abelianAreNeron}. 
These results give a partial answer to the general existence problem for N\'eron models of abelian varieties over arbitrary valuation rings.

The bridge between this algebraic result and formal geometry over
$K^+$ is a gluing statement for admissible formal schemes. We denote the corresponding category by $\operatorname{Adm}$. Namely, we prove in \cref{p:pushout_equivalence}
that pullback induces an equivalence
\[
\begin{split}
    \operatorname{Adm}/\operatorname{Spf}(K^+)
    \cong {}&
    \bigl(\operatorname{Adm}/\operatorname{Spf}(K^\circ)\bigr)
    \\
    &\times_{
      \operatorname{Schemes}^{\mathrm{flat,lfp}}/
      \operatorname{Spec}(k)}
    \bigl(
      \operatorname{Schemes}^{\mathrm{flat,lfp}}/
      \operatorname{Spec}(K^+/K^{\circ\circ})
    \bigr),
\end{split}
\]
compatible, among other things, with smooth and proper morphisms. {Here $k$ is the residue field of $K^\circ$.}
In other words, the datum of an object over $\operatorname{Spf}(K^+)$ is equivalent to that of a pair of objects over $\operatorname{Spf}(K^\circ)$ and
$\operatorname{Spec}(K^+/K^{\circ\circ})$ together with a gluing datum over $\Spec(k)$.
The reduction over
$\operatorname{Spec}(K^+/K^{\circ\circ})$ is a scheme over a valuation ring to which we can now apply our algebraic results.

We use this gluing result to construct a formal marked N\'eron model for an abeloid variety $A$ over $\SKKp$ in the sense of \cref{d:markedNeron_K+} after a finite separable base change.
We start with a
marked N\'eron model of $A^\circ$ over
$\operatorname{Spf}(K^\circ)$, its special fiber is semi-abelian.
Since $K^+/K^{\circ\circ}$ is a valuation ring with fraction field
$k$, \cref{l:semi-abelian-model} provides, up to a finite
separable base field extension, a marked N\'eron model of
this semi-abelian variety over
$\operatorname{Spec}(K^+/K^{\circ\circ})$. Gluing the two models
produces a formal marked N\'eron model over $\operatorname{Spf}(K^+)$.

In this way we obtain the following result in \cref{t:uniformization_formal1}:
\begin{introtheorem}\label{t:intro3}
    Let $A$ be an abeloid variety over $\SKKp$.
    After a finite separable affinoid field extension, $A$ 
    admits a marked Néron model $\mathcal{E}$ that fits into a short exact sequence
        % https://q.uiver.app/#q=WzAsNSxbMSwwLCJcXGNUIl0sWzIsMCwiXFxvdmVybGluZXtcXHJlc3Ryc21hbGx7RX17XFxjQlxcZ2VuZn19Il0sWzMsMCwiXFxjQiJdLFswLDAsIjAiXSxbNCwwLCIwIl0sWzMsMF0sWzAsMV0sWzEsMl0sWzIsNF1d
        \[\begin{tikzcd}
        	0 & \cT & {\cE} & \cB & 0
        	\arrow[from=1-1, to=1-2]
        	\arrow[from=1-2, to=1-3]
        	\arrow[from=1-3, to=1-4]
	        \arrow[from=1-4, to=1-5]
        \end{tikzcd}\]
        where $\cT$ is a split formal torus and $\cB$ is a formal semi-abelian scheme whose $\SKc$-fiber is a formal abelian scheme.
        
        If $k$ is algebraically closed or $K^+/K^{\circ\circ}$ is a discrete valuation ring, then, after such a base change, $A$ admits a strong marked Néron model, and $\cE$ is its fiberwise identity component.
\end{introtheorem}

When $k$ is algebraically closed we show in \cref{p:Neron_finite_m} that multiplication by $m$ is finite and faithfully flat on strong marked Néron models. We use this to deduce the analogous statement for Néron models of semi-abelian varieties over algebraically closed fields. The proof argues by passing to a microbial valuation ring of strictly higher rank, where our generalized uniformization theory applies. In contrast, this statement generally fails in the discretely valued setting.

We proceed to investigate the connection of the marked Néron model to the Raynaud uniformization over $\SKKp$ in \cref{c:uniformization_formal2},
which is slightly more subtle than in rank one.
Pushing out the generic fiber of the formal semi-abelian extension along the inclusion of its formal torus into the full algebraic torus
defines an open part of the Raynaud extension of $A$.  Its $\SKKc$-fiber is the classical Raynaud extension.
The full Raynaud extension arises directly in this way from the
marked N\'eron model if and only if
the generically good reduction part $B$ has good reduction. More precisely, we obtain the following result in \cref{c:uniformization_formal3}:

\begin{introcorollary}\label{t:intro4}
    Let $A$ be an abeloid variety over $\SKKp$. After a base change as in \cref{t:intro3}, the following are equivalent:
    \begin{enumerate}
        \item The generically good reduction part $B$ in the Raynaud uniformization of $A$ has good reduction.
        \item The Raynaud extension of $A$ arises as the pushout of the generic fiber of its marked Néron model.
        \item The fiberwise torus rank of the marked Néron model is constant (i.e., $\cB$ is a formal abelian scheme).
    \end{enumerate}
\end{introcorollary}

We briefly describe the organization of the paper. 
In
\cref{s:2} we recall the necessary background on valuation rings, adic
spaces, formal models, and classical Raynaud uniformization.  In
\cref{s:3} we compare spaces over $\operatorname{Spa}(K,K^+)$ with
their rank one fibers.  \cref{s:4}
uses these results to construct the higher rank Raynaud uniformization
and studies dual abeloid varieties.  \cref{s:5} considers N\'eron models
of semi-abelian varieties over arbitrary valuation rings.  Finally, in \cref{s:6} we apply these
algebraic results to construct formal N\'eron models over
$\operatorname{Spf}(K^+)$, compare them with Raynaud uniformization, and show that $[m]$ is finite faithfully flat on Néron models {in the algebraically closed setting}.
We also discuss the complete picture for elliptic curves.

\subsection*{Acknowledgements} We would like 
to thank Ben Heuer for the suggestion to study Raynaud uniformization of abeloid varieties over general affinoid fields via compactification. We are also grateful to Lucas Gerth for his useful comments on our manuscript and for suggesting to look at torsion points. We thank Katharina H\"ubner for suggesting to the first author to get rid of a finite type condition for the inclusion $K^+\subseteq K^\circ$. 
This project has received funding from the Deutsche Forschungsgemeinschaft (DFG, German Research Foundation) TRR 326 \emph{Geometry and Arithmetic of Uniformized Structures}, project number 444845124. 

\section{Background, notation and conventions}\label{s:2}

\subsection{Valuation rings and affinoid fields}

In this paper, we work over a non-Archimedean field $K$, i.e., over a field $K$ which is complete with respect to a non-trivial non-Archimedean absolute value. The associated rank one valuation ring is denoted by $K^\circ$ and the residue field by $k$. We denote by $\varpi\in K^{\circ\circ}=\mathfrak{m}_{K^\circ}$ a pseudo-uniformizer, and $K^+\subseteq K^\circ$ always denotes a valuation ring of $K$. A pair of the form $(K, K^+)$ is an \textit{affinoid field}.  An \textit{affinoid field extension} of $(K, K^+)$ is an affinoid field $(L, L^+)$ and a morphism of Huber pairs $(K, K^+)\to (L, L^+)$ such that $L^+\cap K=K^+$. It is called finite (separable, respectively) if $L/K$ is finite (separable, respectively). Our main objects of study are adic spaces over a base $\SKKp$.

\begin{example}\label{e:affinoid_field}
Let us consider examples of higher rank valuation rings and affinoid fields.
\begin{enumerate}[leftmargin=2.5em]
\item Let $x\in X$ be a point of an analytic adic space. The completed residue field $(\kappa(x),\kappa(x)^+)$ at $x$ is an affinoid field, and the height of $\kappa(x)^+$ is the codimension of $x$ plus $1$. If $Y\to X$ is a morphism of adic spaces, then
\begin{equation*}
    Y_x:=Y\times_X \Spa(\kappa(x),\kappa(x)^+)
\end{equation*}
defines an adic space over $ \Spa(\kappa(x),\kappa(x)^+)$ under mild conditions (e.g., when $Y\to X$ is locally of weakly topologically finite type). 

\item More explicitly, consider the adic closed unit disc
\begin{equation*}
    \dB_{\dC_p}:=\Spa(\dC_p\langle \hspace{0.15mm}t\hspace{0.15mm}\rangle, \dC_p^\circ\langle \hspace{0.15mm}t\hspace{0.15mm}\rangle).
\end{equation*}
The Gauss point $g\in \dB_{\dC_p}$ is given by the valuation $|\sum_n a_n t^n|_{g} := \max_n |a_n|_p$.  Let $x\in \dB_{\dC_p}$ be a height two point generalizing to $g$, i.e., a type $5$ point of $g$ (see, e.g., \cite[Lecture 11]{Conrad2015PerfectoidSeminar}). Then $x$ is given by a rank $2$ valuation on $K := \Frac(\dC_p \langle \hspace{0.15mm}t\hspace{0.15mm} \rangle)^\wedge$, where we complete with respect to the Gauss norm. If $K^+$ denotes the associated valuation ring, then $(K, K^+)$ is an affinoid field whose adic spectrum has two points.

\item  Consider $K:=\dQ_p((t))$ with the $t$-adic topology. We equip the residue field $\dQ_p$ with the $p$-adic 
valuation.
This induces the rank $2$ valuation 
    \begin{equation}\label{eq:valuation_example}\tag{$*$}
        K^\times \ni f=\sum_{j\geq j_0, \,f_{j_0}\neq 0} f_j t^j\longmapsto (\lvert f\rvert_t, \lvert f_{j_0}\rvert_p)\in \Gamma_{{\lvert \,\cdot\,\rvert_t}}\times \Gamma_{\lvert \,\cdot\,\rvert_p},
    \end{equation} 
     where the value group carries the lexicographic order. Its valuation ring is
     \begin{equation*}
         K^+=\dZ_p + t\dQ_p[[t]].
     \end{equation*}
     Moreover, the residue field is $k=\dQ_p$, and $K^+/{K^{\circ\circ}}=\dZ_p$. This gives another example of $\SKKp$ of rank 2. The valuation (\ref{eq:valuation_example}) is the closed point of $\SKKp$, and the generic point is given by the $t$-adic valuation. More generally, this construction can be carried out for every
     valued field
     in place of $\dQ_p$.
    \end{enumerate}
\end{example}

Note that every continuous valuation $v$ on $K$ which is at most one on $K^+$ gives rise to a valuation ring  $R$ with $K^+ \subset R \subset K^\circ$. These valuation rings form a totally ordered set by inclusion, and they are all equal to  a localization $K^+_{\mathfrak{p}}$ by some prime ideal $\mathfrak{p}$ in $K^+$. In fact, localization induces 
a  homeomorphism
\[\Spec(K^+)\setminus \left\{(0)\right\} \longrightarrow \SKKp .\]
Hence $\SKKp$ is totally ordered, its maximal element is the generic point $v_{K^\circ}$, and its minimal element is $v_{K^+}$. In particular, the generic point is the image of the natural map $\SKKc \to \SKKp$. Moreover, let us recall that the map
\begin{equation*}
    \{R\subseteq K^\circ\mid R\text{ is a valuation ring of } K\} \longrightarrow \{\text{valuation rings of } k\},\quad R\longmapsto R/K^{\circ\circ},
\end{equation*}
is bijective. The inverse is given by mapping a valuation ring $V\subseteq k$ to the preimage of $V$ under the reduction map $K^\circ \to k$.

\subsection{Conventions}
Throughout this paper, $(K, K^+)$ will always denote an affinoid field, $\varpi\in K^+$ denotes a pseudo-uniformizer, and $k$ denotes the residue field $K^\circ/K^{\circ\circ}$. We use the following conventions throughout.

\begin{convention}\label{convention_adic} 
All adic spaces are assumed to be analytic and locally strongly Noetherian in the sense of \cite[Section 1.1]{Huber1996EtaleCohomology}. Additionally, whenever we write down a Huber pair $(A, A^+)$, we implicitly assume that it is a complete strongly Noetherian Tate pair. 
\end{convention}
For every such adic space $X$ we  write $\mathrm{Adic}_X^{\mathrm{P}}$ for the
full subcategory
of all such adic spaces $Y$ together with a morphism $Y \to X$ which has property $\mathrm{P}$.

Quotients and short exact sequences of adic group spaces are always understood as sheaves on the analytic site.

Recall that under \cref{convention_adic}, all morphisms of adic spaces are adic (hence spectral), and a fiber product of adic spaces exists when one of the maps is locally of weakly topologically finite type \cite[Prop.\ 1.2.2]{Huber1996EtaleCohomology}.

\begin{convention}\label{convention_formal}
By a formal scheme we always mean an adic formal scheme of finite ideal type in the sense of \cite[Def.\ I.1.1.16]{FujiwaraKato2018}.
\end{convention}

We call a formal scheme $\cX\to \SKp$ \textit{admissible} if it is flat and locally of finite presentation over $\SKp$, see \cite[Def.\ I.2.3.3]{FujiwaraKato2018}.

The construction of the \textit{adic generic fiber} that has classically been considered over $\SKc$ can also be carried out over $\SKp$, see \cite{Dine2025FormalModels}. More concretely, an affinoid admissible formal scheme $\Spf(A)\to \SKp$ has the adic generic fiber $\Spa(A[\varpi^{-1}],\intcl{A})\to \SKKp$, 
{where $\intcl{A}$ is the integral closure in $A[\varpi^{-1}]$.} This glues to a functor $(-)\genf$ from the category of admissible formal schemes over $\SKp$ to the category of adic spaces locally of topologically finite type over $\SKKp$.
Moreover, for every admissible formal scheme $\cX\to \SKp$ there is a natural surjective, closed, and continuous \textit{specialization map} of topologically ringed spaces $\mathrm{sp}_\cX\colon \cX\genf\to\cX$, see \cite[Cor.\ 4.16, 4.20]{Dine2025FormalModels}.
A \textit{formal model} of an adic space $X\to \SKKp$ is an admissible formal scheme $\cX\to \SKp$ together with an isomorphism $\cX\genf\cong X$ over $\SKKp$. When $X$ and $\cX$ carry structures of group spaces and if the isomorphism respects these structures, then $\cX$ is called a {formal} \textit{group model} of $X$.

A \textit{smooth} (\textit{proper}, respectively) formal scheme over $\SKp$ is an admissible formal scheme $\cX\to \SKp$ whose \textit{reduction} $\cX \times_{\SKp} \Spec(K^+/K^{\circ\circ})\to \Spec(K^+/K^{\circ\circ})$ is smooth (proper, respectively), see \cite[Prop.\ I.5.3.18, Prop.\ I.4.7.3]{FujiwaraKato2018} and \spcite{0CF4}. A \textit{formal abelian scheme} over $\SKp$ is a smooth proper formal group space $\cA\to \SKp$ with connected fibers. This is equivalent to being an admissible group space whose reduction over $\Spec(K^+/{K^{\circ\circ}})$ is an abelian scheme, see \cite[Prop.\ I.4.7.3]{FujiwaraKato2018} and \spcite{0CF4}.

\subsection{Rigid spaces} We recall that there is an equivalence of categories $r$ from the category of quasi-separated rigid spaces over $\mathrm{Sp}(K)$ in the sense of Tate to the category of quasi-separated adic spaces locally of topologically finite type over $\SKKc$
in the sense of Huber
(cf.\ \cite[(1.1.11)]{Huber1996EtaleCohomology}). We will use it freely to translate the classical theory to the language of adic spaces.

It preserves and reflects many properties of morphisms, e.g., quasi-compactness, being a closed immersion, and smoothness; see \cite[Prop.\ 3.11.21]{Huber1993Bewertungsspektrum} and \cite[1.7.11]{Huber1996EtaleCohomology}. It also preserves and reflects properness: this follows from \cite[Cor.\ 4.4, 4.5]{Temkin2000LocalProperties}, settling the question whether properness of a formal model implies properness of the rigid generic fiber. In general the equivalence does not reflect the property of being an open immersion, but it does preserve and reflect the property of being a quasi-compact open immersion into a quasi-separated space, see \cite[(1.1.11)]{Huber1996EtaleCohomology}.
Moreover, for every rigid space $X\to\mathrm{Sp}(K)$, $r$ induces an equivalence of topoi \begin{equation*}
    \mathrm{Sh}(X_{\mathrm{adm}})\cong \mathrm{Sh}(r(X)_\mathrm{an}).
\end{equation*}

\subsection{Raynaud uniformization over rank one valued 
fields}

We first recall the classical picture where the valuation ring $K^+=K^\circ$ is of rank one.

\begin{definition}
    An abeloid variety over $\SKKc$ is a proper, smooth, geometrically connected group space $A\to \SKKc$. We say that it has good reduction if it admits a formal group model over $\SKc$ that is a formal abelian scheme.
\end{definition}

We can view the following theorem as the non-Archimedean counterpart to the well-known uniformization of complex tori as quotients of a complex vector space by a lattice. 

\begin{theorem}[\mbox{\cite[Cor.\ 6.3.4, Thm.\ 7.6.4]{Luetkebohmert2016RigidGeometryCurves}}]\label{t:classical_uniformization_analytic}
   Let $A$ be an abeloid variety over $\SKKc$. After a finite separable field extension, there exists an isomorphism $A\cong E/M$, where $0 \to T\to E\to B \to 0$ is an extension of an abeloid variety $B$ with good reduction by a split algebraic torus $T\cong (\dG_{m,K}^{\mathrm{an}})^r$, and $\underline{\dZ^r}\cong M\subseteq{} E$ is a closed subgroup.
   
   In other words, we have a diagram of short exact sequences of adic group spaces over $\SKKc$ which we call the \textit{uniformization cross}\footnote{In fond remembrance of Robert Coleman, who used this terminology in conversations with the second author.} of $A\cong E/M$:
        % https://q.uiver.app/#q=WzAsOSxbMSwyLCJUIl0sWzIsMiwiRSJdLFszLDIsIkIiXSxbMCwyLCIwIl0sWzQsMiwiMCJdLFsyLDEsIk0iXSxbMiwzLCJBIl0sWzIsMCwiMCJdLFsyLDQsIjAiXSxbMCwxXSxbMSwyXSxbMywwXSxbMiw0XSxbMSw2XSxbNSwxXSxbNyw1XSxbNiw4XV0=
\[\begin{tikzcd}[column sep=0.9em,row sep=0.7em]
	&& 0 && \\
	&& M \\
	0 & T & E & B & 0. \\
	&& A \\
	&& 0
	\arrow[from=1-3, to=2-3]
	\arrow[from=2-3, to=3-3]
	\arrow[from=3-1, to=3-2]
	\arrow[from=3-2, to=3-3]
	\arrow[from=3-3, to=3-4]
	\arrow[from=3-3, to=4-3]
	\arrow[from=3-4, to=3-5]
	\arrow[from=4-3, to=5-3]
\end{tikzcd}\]
The horizontal extension is called the Raynaud extension of $A$.

Moreover, the Raynaud uniformization $E\to A$ is the universal (analytic) covering space of $A$ in the sense that $H^1(E_{\mathrm{an}}, \dZ)=0$ and that $E\to A$ is final among pointed morphisms $S\to A$ where $S$ is a connected adic space with $H^1(S_{\mathrm{an}}, \dZ)=0$. Every homomorphism between abeloid varieties that admit a uniformization cross uniquely lifts to their uniformization cross. 
\end{theorem}

Note that the universal property of $E$ generalizes from rigid spaces to adic spaces. This follows from the long exact sequence coming from $0\to \underline{\dZ^r}\to E\to A\to 0$. 

The theory of Raynaud uniformization is connected to the existence of suitable
formal models. If $K^\circ$ is a discrete valuation ring, and $K$ is big enough so that the N\'eron model of $A$ has semi-stable reduction, we may obtain $E$ by lifting its special fiber, see \cite{BoschLuetkebohmert1991Degenerating}. For general rank one valuation rings, we still have a marked N\'eron model constructed by Lütkebohmert \cite[Section 7.6]{Luetkebohmert2016RigidGeometryCurves}. See \mbox{\cite[ Rem.\ 7.6.8]{Luetkebohmert2016RigidGeometryCurves}} or \cref{d:markedNeron_K+} for the definition of a marked Néron model.

\begin{theorem}[\mbox{\cite[Prop.\ 7.6.3, Thm.\ 7.6.4, Rem.\ 7.6.8]{Luetkebohmert2016RigidGeometryCurves}, \cite[p.\ 656]{BoschLuetkebohmert1991Degenerating}}]\label{t:classical_uniformization_formal1}
   Let $A$ be an abeloid variety over $\SKKc$. After a finite separable field extension, $A$ admits a marked Néron model
   $\cE\to \SKc$ that fits into a short exact sequence
    % https://q.uiver.app/#q=WzAsNSxbMSwwLCJcXGNUIl0sWzIsMCwiXFxvdmVybGluZXtcXHJlc3Ryc21hbGx7RX17XFxjQlxcZ2VuZn19Il0sWzMsMCwiXFxjQiJdLFswLDAsIjAiXSxbNCwwLCIwIl0sWzMsMF0sWzAsMV0sWzEsMl0sWzIsNF1d
\[\begin{tikzcd}
	0 & \cT & {\cE} & \cB & 0
	\arrow[from=1-1, to=1-2]
	\arrow[from=1-2, to=1-3]
	\arrow[from=1-3, to=1-4]
	\arrow[from=1-4, to=1-5]
\end{tikzcd}\]
    where $\cB$ is a formal abelian scheme and $\cT\to \SKc$ is a split formal torus. Then
    \begin{enumerate}
        \item let $\cT\genf\to T$ be the inclusion in its full algebraic torus, which is isomorphic to a power of $\dG_{m,K}^{\mathrm{an}}$. The pushout of the generic fiber
    % https://q.uiver.app/#q=WzAsMTAsWzEsMCwiXFxjVFxcZ2VuZiJdLFszLDAsIlxcY0JcXGdlbmYiXSxbMSwxLCJUIl0sWzIsMSwiRSJdLFszLDEsIlxcY0JcXGdlbmYiXSxbMCwwLCIwIl0sWzAsMSwiMCJdLFs0LDAsIjAiXSxbNCwxLCIwIl0sWzIsMCwiXFxjRVxcZ2VuZiJdLFswLDJdLFsyLDNdLFszLDRdLFsxLDQsIiIsMCx7ImxldmVsIjoyLCJzdHlsZSI6eyJoZWFkIjp7Im5hbWUiOiJub25lIn19fV0sWzMsMCwiIiwwLHsic3R5bGUiOnsibmFtZSI6ImNvcm5lciJ9fV0sWzUsMF0sWzYsMl0sWzEsN10sWzQsOF0sWzAsOV0sWzksMV0sWzksM11d
\[\begin{tikzcd}[row sep=small]
	0 & {\cT\genf} & {\cE\genf} & {\cB\genf} & 0 \\
	0 & T & E & {\cB\genf} & 0
	\arrow[from=1-1, to=1-2]
	\arrow[from=1-2, to=1-3]
	\arrow[from=1-2, to=2-2]
	\arrow[from=1-3, to=1-4]
	\arrow[from=1-3, to=2-3]
	\arrow[from=1-4, to=1-5]
	\arrow[equals, nfold, 
    from=1-4, to=2-4]
	\arrow[from=2-1, to=2-2]
	\arrow[from=2-2, to=2-3]
	\arrow["\lrcorner"{anchor=center, pos=0.125, rotate=180}, draw=none, from=2-3, to=1-2]
	\arrow[from=2-3, to=2-4]
	\arrow[from=2-4, to=2-5]
\end{tikzcd}\]
is the Raynaud extension of $A$.
    \item The open immersion $\cE\genf\subseteq A$
    uniquely extends as
    a homomorphism $\cE\genf\subseteq E\to A$, and the latter map is the Raynaud uniformization of $A$.
    \end{enumerate}
\end{theorem}

For an abeloid variety $A\to \SKKc$, the presheaf of translation-invariant line bundles (cf.\ \cref{d:dual_abeloid}) is representable by an abeloid variety $A^\vee\to \SKKc$ called the \textit{dual abeloid variety} of $A$. This is \cite[Cor.\ 7.6.5]{Luetkebohmert2016RigidGeometryCurves} and a consequence of \cref{t:classical_uniformization_formal1}. More precisely, when $A$ admits a Raynaud uniformization over $\SKKc$, then one has an explicit description of $A^\vee$ in terms of the dual of its good reduction part $B$. There is a close relation between the uniformization crosses of $A$ and $A^\vee$, see
\cite[Thm.\ 6.3.3]{Luetkebohmert2016RigidGeometryCurves}.

\section{\texorpdfstring{From $\SKKp$ to $\SKKc$ and back}
    {From Spa(K,K⁺) to Spa(K,K⁰) and back}}\label{s:3}

{
\subsection{\texorpdfstring{Descending properties from the $\SKKc$-fiber}
    {Descending properties from the Spa(K,K⁰)-fiber}}
Recall our convention \ref{convention_adic} that all adic spaces are analytic and locally strongly Noetherian.
For an adic space $X\to \SKKp$ we denote the base change with $\SKKc$ by 
$X^\circ:=X\times_{\SKKp} \SKKc$.
It exists as $\SKKc\to\SKKp$ is locally of weakly topologically finite type (cf.\ \cite[Prop.\ 1.2.2]{Huber1996EtaleCohomology}).
For example, when $X=\Spa(R, R^+)$ is an affinoid space {over $\SKKp$}, then 
\begin{equation}\label{eq:affinoid_circ}
    X^\circ=\Spa(R, R^+[K^\circ])
\end{equation}
and the topology on $R$ is unchanged. Note that $R^+[K^\circ]=(K^+\cap (K^\circ)^\times)^{-1}R^+$. It is therefore integrally closed in $R$.
The space $X^\circ$ is often classically better understood than $X$. For example, $X^\circ$ has a natural counterpart in the category of Tate's rigid spaces if $X$ is locally of topologically finite type and quasi-separated. We prove some basic observations that relate objects and morphisms of adic spaces over $\SKKp$ to their $\SKKc$-fiber.
}

By \cite[Def.\ 1.3.3]
{Huber1996EtaleCohomology}, a morphism $f$ between
adic spaces is called partially proper if it is locally of ${}^+\!$weakly topologically finite type in the sense of \cite[Def.\ 1.2.1]{Huber1996EtaleCohomology}, separated, and
universally specializing. It is proper if 
it is quasi-compact and partially proper.
A proper morphism is universally closed.

\begin{example}
	Let $X$ be a separated $K$-scheme locally of finite type. Then the associated analytification $X^\mathrm{an}\to \SKKp$ is partially proper. If $X$ is proper, then $X^\mathrm{an}$ is proper (cf.\ \cite[Lem.\ 5.7.3]{Huber1996EtaleCohomology}).
    
    Moreover, note that the universal property of analytification implies that $(X^{\mathrm{an}})^\circ\to \SKKc$ is the analytification of $X$ along $\SKKc$.
\end{example}

Moreover, we have  the following valuative criterion.

\begin{lemma}[\mbox{\cite[Lem.\ 1.3.10]{Huber1996EtaleCohomology}, \cite[Prop.\ 6.13]{Ludwig2020AdicSpacesII}}]\label{l:valuative criterion}
    Let $f\colon X\to S$ be a quasi-separated and locally ${}^+\!$weakly topologically finite type morphism of
    adic spaces. Then $f$ is partially proper if and only if for every affinoid space $\Spa(R, R^+)\to S$, the natural map
    \[\mathrm{Mor}_S(\Spa(R, R^+),X)\longrightarrow \mathrm{Mor}_S(\Spa(R, R^\circ),X)\]
    is bijective.
\end{lemma}

Since $(-)^\circ$ only changes the $+$-component of the structure sheaf and does so in a functorial and continuous way, one deduces the following statement in our setting.

\begin{proposition}\label{p:valuative_strong_criterion}
    Let $S\to \SKKp$ be an
    adic space and suppose that $X\to S$ is partially proper (separated, respectively). For every 
    adic space $T\to S$, the canonical and natural maps
    \begin{equation*}
        \mathrm{Mor}_S(T,X)\longrightarrow \mathrm{Mor}_{S^\circ}(T^\circ,X^\circ)\longrightarrow \mathrm{Mor}_S(T^\circ,X)
    \end{equation*}
    are bijective (injective, respectively).
\end{proposition}
\begin{proof}
    It follows from the universal property of fiber products that the second map is bijective. Let us now consider the first map. By gluing, we may assume that $T$ is affinoid. The statement for the partially proper case now follows from \cref{eq:affinoid_circ,l:valuative criterion}.
    
    In the separated case, suppose that $f,g\colon T\to X$ are such that $f^\circ=g^\circ$. The equalizer $\mathrm{Eq}(f,g)\to T$ is a closed immersion, hence in particular partially proper. The partially proper case of this proposition implies that $(-)^\circ$ is fully faithful on the category of partially proper $T$-spaces. Since $(\mathrm{Eq}(f,g))^\circ\to T^\circ$ is an isomorphism, it follows that $\mathrm{Eq}(f,g)\to T$ is an isomorphism as well.
\end{proof}

\begin{remark}\label{r:circ_fully_faithful}
	Let $X\to \SKKp$ be an adic space. \Cref{p:valuative_strong_criterion} shows that 
    \begin{equation*}
	    (-)^\circ\colon \mathrm{Adic}_X^{\mathrm{p. proper}}\longrightarrow \mathrm{Adic}_{X^\circ}^{\mathrm{p. proper}}
	\end{equation*}
    is fully faithful. When $\SKKc\to \SKKp$ is of ${}^+\!$weakly topologically finite type, then it follows from Huber's universal compactification theorem \cite[Thm.\ 5.1.5]{Huber1996EtaleCohomology} that the functor is an equivalence. 
    This assumption is equivalent to the existence of
    an $e \in K^+ \setminus K^{\circ\circ}$ such that $K^\circ = K^+[e^{-1}]$. However, we will not need it. 
\end{remark}

\begin{lemma}\label{l:circ_embedding}
    Let $X\to \SKKp$ be a morphism of
    adic spaces. Then $j\colon X^\circ \to X$ is a quasi-compact topological embedding, every element of $X$ is a specialization of an element of $X^\circ$, and $\cO_X \to j_*\cO_{X^\circ}$ is an isomorphism. Moreover,
    \begin{enumerate}[label=(\arabic*)]
    \item $j_*\colon \mathrm{Coh}(X^\circ)\to \mathrm{Coh}(X)$ is an equivalence with inverse $j^{-1}$.
    \item The square
    % https://q.uiver.app/#q=WzAsNCxbMCwxLCJcXG1hdGhybXtBZGljfV97WH1eXFxtYXRocm17ZmluaXRlfSJdLFsxLDEsIlxcbWF0aHJte0FkaWN9X3tYXlxcY2lyY31eXFxtYXRocm17ZmluaXRlfSJdLFswLDAsIlxcbWF0aHJte0NvaEFsZ30oWClee1xcbWF0aHJte29wfX0iXSxbMSwwLCJcXG1hdGhybXtDb2hBbGd9KFheXFxjaXJjKV57XFxtYXRocm17b3B9fSJdLFsyLDMsImpeey0xfSJdLFsyLDAsIlxcdW5kZXJsaW5le1xcU3BhfV97WH0iLDJdLFszLDEsIlxcdW5kZXJsaW5le1xcU3BhfV97WF5cXGNpcmN9IiwyXSxbMCwxLCIoLSlee1xcY2lyY30iXV0=
\[\begin{tikzcd}[row sep=scriptsize]
	{\mathrm{CohAlg}(X)^{\mathrm{op}}} & {\mathrm{CohAlg}(X^\circ)^{\mathrm{op}}} \\
	{\mathrm{Adic}_{X}^\mathrm{finite}} & {\mathrm{Adic}_{X^\circ}^\mathrm{finite}}
	\arrow["{j^{-1}}", from=1-1, to=1-2]
	\arrow["{\underline{\Spa}_{X}}"', from=1-1, to=2-1]
	\arrow["{\underline{\Spa}_{X^\circ}}"', from=1-2, to=2-2]
	\arrow["{(-)^{\circ}}", from=2-1, to=2-2]
\end{tikzcd}\]
    commutes up to natural isomorphism. All four functors are equivalences.
    \end{enumerate}
\end{lemma}
\begin{proof}
    The morphism $j$ is quasi-compact by base change. It is a topological embedding as it is the inclusion of the $\SKKc$-fiber. Since every morphism of analytic adic spaces maps rank one points to rank one points, $X^\circ$ contains all rank one points of $X$ (cf.\ \cite[Lem.\ 1.1.10] {Huber1996EtaleCohomology}). Therefore, every element of $X$ is a specialization of an element of $X^\circ$.
    
    (1) follows as in the proof of \cite[Lem.\ 3.1]{Mann2023NormalIrreducibleAdicSpaces}. Indeed, the isomorphism $\cO_X \to j_*\cO_{X^\circ}$ is immediate from \cref{eq:affinoid_circ}, since it suffices to consider sections of the sheaves on affinoid opens. Recall that the category of coherent sheaves on an affinoid adic space $\Spa(A,A^+)$ is equivalent to the category of finite $A$-modules, see \cite[Satz.\ 3.3.18]{Huber1993Bewertungsspektrum}. Hence the isomorphism $\cO_X \cong j_*\cO_{X^\circ}$ together with the fact that the preimage under $j$ of an affinoid open is affinoid implies that $j_*\colon \mathrm{Coh}(X^\circ)\to \mathrm{Coh}(X)$ is an equivalence. (2) is a consequence of (1), see \cite[Satz.\ 3.12.12]{Huber1993Bewertungsspektrum} for the vertical arrows.
\end{proof}

The following lemma records a useful topological relation between an adic space over $\SKKp$ and its rank one fiber.

\begin{lemma}\label{l:dense_cocont}
    Let $j\colon X\to X'$ be a quasi-compact injective morphism of
    adic spaces such that every point of $X'$ is a specialization of a point of $X$. Suppose that $X=\bigcup U_i$ is an open cover and set $U_i^\mathrm{sat}:=\mathrm{int}((X'\setminus X) \cup U_i)$. Then $X'=\bigcup U_i^\mathrm{sat}$ and $U_i^\mathrm{sat}\cap X = U_i$.
\end{lemma}
\begin{proof}
    Note that $j$ is a homeomorphism onto its image since it is an injective quasi-compact map between analytic adic spaces, see \cite[Lem.\ 1.3.15]{Huber1996EtaleCohomology}. One verifies that $U_i^\mathrm{sat}= X'\setminus\overline{X\setminus U_i}$ and $U_i^\mathrm{sat}\cap X = U_i$.
    
    Now, suppose toward a contradiction that there exists a point $x' \in X' \setminus \bigcup U_i^\mathrm{sat} = \bigcap \overline{X \setminus U_i}$. Let $V\subseteq X'$ be an affinoid neighborhood of $x'$. Then $\operatorname{Spa} (\kappa(x'), \kappa(x')^+)\xrightarrow{\varphi}V\to X'$ factors through $V$, as it is a homeomorphism onto the set of generalizations of $x'$. Therefore $\varphi^{-1}(V\cap X)$ is quasi-compact as $X\to X'$ is quasi-compact.
    It is also nonempty and thus contains a closed point $x$, which is a minimal generalization of $x'$ in $X$.
    
    Let $l$ be such that $x \in U_l$. As $j$ is spectral and quasi-compact, $j(X \setminus U_l) \subseteq X'$ is pro-constructible (cf.\ \cite[(1.1.13)]{Huber1996EtaleCohomology}). 
     Hence the closure of $X \setminus U_l$ in $X'$ consists of all specializations of points in $X \setminus U_l$ (see \spcite{0903} for the spectral case; the general case follows from this). Therefore there exists $y \in X \setminus U_l$ such that $y \succ x'$. As generalizations are totally ordered and $x \in U_l$, it must be that $x \succ y$, contradicting the minimality of $x$.
\end{proof}

\begin{remark}\label{r:cohomology_comparison}
    Let $j\colon X\subseteq X'$ be as in \cref{l:dense_cocont}. Then it follows from \cref{l:dense_cocont} that $j_*\colon \mathrm{Sh}(X_{\mathrm{an}})\to \mathrm{Sh}(X'_{\mathrm{an}})$ preserves epimorphisms.
    In particular, one obtains $H^i(X_{\mathrm{an}}, \cF)=H^i(X'_{\mathrm{an}}, j_*\cF)$ for $i\geq 1$ and abelian sheaves $\cF$. When $\cF$ is a sheaf of (non-commutative) groups, then one similarly verifies the equality for $i=1$.
    
    Therefore, if $X\to \SKKp$ is an adic space
    and $G\to X$ is a partially proper group space over $X$, then $((-)^\circ)_*G^\circ=G$ by \cref{p:valuative_strong_criterion}, and hence \begin{equation*}
        H^i(X^\circ_{\mathrm{an}}, G^\circ)=H^i(X_{\mathrm{an}}, G).
    \end{equation*}
\end{remark}

\begin{lemma}\label{l:circ_reflect_pp}
	Let $f\colon X\to Y$ be a partially proper morphism of adic spaces over $\SKKp$. Then $f$ is finite (a closed immersion, proper, surjective, an epimorphism of sheaves on the analytic site, respectively) if and only if $f^\circ$ has the respective property.
\end{lemma}
\begin{proof}
    The functor $(-)^\circ$ preserves these properties since it is a base change.  The converse statement for \textit{finite} and \textit{closed immersion} is a consequence of 
{\cref{p:valuative_strong_criterion,l:circ_embedding}}. For \textit{proper} we may assume that $Y$ is quasi-compact and quasi-separated. Then \cite[Lem.\ 1.3.13]{Huber1996EtaleCohomology} says that since $f$ is partially proper, the closure $\overline{X^\circ}=X$ is quasi-compact. Hence $f$ is quasi-compact and therefore proper. \textit{Surjectivity} follows readily, as a partially proper map is specializing.
For the last property, suppose that $f^\circ$ is an epimorphism of sheaves on the analytic site, i.e., $f^\circ$ admits sections locally on $Y^\circ$. A cover of $Y^\circ$ can be extended to a cover of $Y$ in the sense of \cref{l:dense_cocont}. Then, \cref{p:valuative_strong_criterion} shows that the sections over $\SKKc$ lift to sections over $\SKKp$.
\end{proof}

\begin{lemma}\label{l:circ_reflects_tft_smooth}
    Let $f\colon X\to Y$ be a morphism that is locally of ${}^+\!$weakly topologically finite type between
    adic spaces over $\SKKp$.
    \begin{enumerate}
        \item $f$ is locally of topologically finite type (flat, étale, smooth, respectively) if and only if $f^\circ$ has the respective property.
        \item If $X$ is separated, then $f$ is an open immersion if and only if $f^\circ$ is.
    \end{enumerate}
\end{lemma}
\begin{proof}
    The morphism $f^\circ$ inherits these properties by base change. We show that $(-)^\circ$ reflects them as well.
    
    (1) For {\textit{flatness}}, assume that $\cO_{Y,f(z)}\to \cO_{X,z}$ is flat for all rank one points $z\in X$. Let $x\in X$ be arbitrary and let $x'$ be its rank one generalization. Then $f(x')$ is the rank one generalization of $f(x)$ and we obtain a commutative square
    % https://q.uiver.app/#q=WzAsNCxbMCwwLCJcXGNPX3tZLFxcLCBmKHgpfSJdLFswLDEsIlxcY09fe1ksIFxcLGYoeCcpfSJdLFsxLDAsIlxcY09fe1gsXFwsIHh9Il0sWzEsMSwiXFxjT197WCxcXCwgeCd9Il0sWzAsMl0sWzEsM10sWzAsMV0sWzIsM11d
\[\begin{tikzcd}[row sep=small]
	{\cO_{Y,\, f(x)}} & {\hphantom{.}\cO_{X,\, x}\hphantom{.}} \\
	{\cO_{Y, \,f(x')}} & {\hphantom{.}\cO_{X,\, x'}.}
	\arrow[from=1-1, to=1-2]
	\arrow[from=1-1, to=2-1]
	\arrow[from=1-2, to=2-2]
	\arrow[from=2-1, to=2-2]
\end{tikzcd}\]
    By \cite[Lem.\ 1.1.10(iii)]{Huber1996EtaleCohomology} the vertical
    maps are local and flat (and hence faithfully flat). The lower map is local and flat by assumption. The flatness of the upper map now follows from \spcite{039V}.
    
We now treat the case where $f^\circ$ is \textit{locally of topologically finite type}.
    By \cref{eq:affinoid_circ}, the claim boils down  to showing that if $f\colon X=\Spa(B, B^+)\to \Spa(A, A^+)=Y$ is such that there exists a quotient map 
    \begin{equation*}
        (A, A^+[K^\circ])\langle X_1,\dots, X_n\rangle\to (B, B^+[K^\circ]),\ X_i\mapsto b_i\in B^+[K^\circ]
    \end{equation*}
    and a finite subset $E = \{r_1,\dots, r_m\} \subseteq B^+$ with $B^+ = \intcl{A^+[B^{\circ\circ} \cup E]}$, then $f$ is of topologically finite type.
    Since $K^+$ is a valuation ring, $K^\circ$ is the filtered union of subrings $K^+[e^{-1}]$ ranging over $e\in K^+\cap (K^\circ)^\times$. Hence $B^+[K^\circ]$ is the filtered union of such $B^+[e^{-1}]$. Therefore, we may choose $e_0\in K^+\cap (K^\circ)^\times$ such that $e_0 b_i \in B^+$ for all $i = 1,\dots, n$. The morphism above can then be factored at the level of topological rings as follows:
    % https://q.uiver.app/#q=WzAsMyxbMCwwLCJBXFxsYW5nbGUgWF8xXFxiaXMgWF9uXFxyYW5nbGUiXSxbMSwxLCJBXFxsYW5nbGUgWF8xJ1xcYmlzIFhfbidcXHJhbmdsZSJdLFsyLDAsIkIiXSxbMCwxLCJYX2lcXCxcXG1hcHN0b1xcLCBlXnstTn1YX2knIiwyXSxbMCwyLCJYX2lcXCwgXFxtYXBzdG9cXCwgIGJfaSJdLFsxLDIsIlhfaSdcXCxcXG1hcHN0b1xcLCBlXntOfWJfaSIsMl1d
\[\begin{tikzcd}[column sep=3.15em,row sep=tiny]
	{A\langle X_1,\dots, X_n\rangle} && {B.} \\
	& {A\langle X_1',\dots, X_n'\rangle}
	\arrow["{X_i\, \mapsto\,  b_i}", from=1-1, to=1-3]
	\arrow["{X_i\,\mapsto\, e_0^{-1} X_i'}"', from=1-1, to=2-2]
	\arrow["{X_i'\,\mapsto\, e_0 b_i}"', from=2-2, to=1-3]
\end{tikzcd}\]
    The left diagonal map is an isomorphism since $e_0\in (K^\circ)^\times$, so the right map is necessarily a quotient morphism of Huber rings. This implies that the induced map
    \begin{equation*}
        (A, A^+)\langle X_1,\dots, X_n, Y_1,\dots, Y_m\rangle\to (B, B^+),\ X_i \mapsto e_0 b_i,\ Y_j\mapsto r_j,
    \end{equation*}
    is a quotient map.
	
    With this, the \textit{étale} case is a consequence of \cite[Prop.\ 1.7.5]{Huber1996EtaleCohomology}, which implies that the property can be verified on rank 1 points if the morphism is locally of topologically finite type.
    
    For \textit{smoothness}, assume that $X^\circ \to Y^\circ$ is smooth. We may assume that $Y$ is affinoid. By \cite[Cor.\ 1.6.10]{Huber1996EtaleCohomology} there exists an open cover $X^\circ=\bigcup U_i$ and étale morphisms $g_i$ such that the following diagrams commute
% https://q.uiver.app/#q=WzAsNCxbMCwwLCJVX2kiXSxbMSwxLCJcXGRCX3tZXlxcY2lyY31ee25faX0iXSxbMiwxLCJcXGRBX3tZXlxcY2lyY31ee25faX0iXSxbMCwyLCJZXlxcY2lyYyJdLFswLDNdLFsxLDMsIlxcbWF0aHJte25hdH0iLDJdLFswLDEsImdfaSIsMl0sWzAsMiwialxcY2lyYyBnX2kiXSxbMiwzLCJcXG1hdGhybXtuYXR9Il0sWzEsMiwiaiIsMV1d
\[\begin{tikzcd}[column sep=large,row sep=small]
	{U_i} \\
	& {\dB_{Y^\circ}^{n_i}} & {\dA_{Y^\circ}^{n_i}.} \\
	{Y^\circ}
	\arrow["{g_i}"', from=1-1, to=2-2]
	\arrow[from=1-1, to=2-3]
	\arrow[from=1-1, to=3-1]
	\arrow["{\mathrm{can}}"{description}, from=2-2, to=2-3]
	\arrow["{\mathrm{can}}"', from=2-2, to=3-1]
	\arrow["{\mathrm{can}}", from=2-3, to=3-1]
\end{tikzcd}\]
    Let $U_i^\mathrm{sat}$ be as in \cref{l:dense_cocont} for the inclusion $X^\circ\subseteq{} X$. Since the morphisms $\dA^{n_i}_Y\to Y$ are partially proper, there exist unique lifts $h_i\colon U_i^\mathrm{sat}\to \dA^{n_i}_Y$ over $Y$ by \cref{p:valuative_strong_criterion}. Since $f$ is locally of topologically finite type by the previous case, the $h_i$ are étale by the previous case. Since the $U_i^\mathrm{sat}$ cover $X$, this concludes the proof.

    (2) Suppose that $f^\circ$ is an open immersion. By the étale case, $X\to Y$ is étale. In particular, by replacing $Y$ by the (open) image of $X$, we may assume that $X\to Y$ is surjective.  It is a monomorphism. Indeed, since $X^\circ\to Y^\circ$ is an isomorphism and $X$ is separated, this follows from \cref{p:valuative_strong_criterion}. Therefore the pullback of $X\to Y$ along itself is an isomorphism. Since $X\to Y$ is a covering in the étale topology, the claim follows from \cite[(2.2.2)]{Huber1996EtaleCohomology}.
\end{proof}

\subsection{\texorpdfstring{Extending spaces to $K^+$}
    {Extending spaces to K⁺}}
We compare admissible formal schemes over $\SKc$ to those over $\SKp$. This will be important for constructing {suitable} Néron models of abeloid varieties over $\SKKp$. Moreover, we use it 
to extend the uniformization cross of the $\SKKc$-fiber of an abeloid variety.

We write $\mathrm{Adm}/\SKp$ for the category of admissible formal schemes over $K^+$, and $\mathrm{Sch}_{\mathrm{flat},\mathrm{lfp}}/\hspace{-0.2mm}S$ for the category of schemes which are flat and locally of finite presentation over $S$.

\begin{proposition}\label{p:pushout_equivalence}
    Pullback defines an equivalence of categories
    \begin{equation*}
    (\mathrm{Adm}/\SKp)\isomarrow (\mathrm{Adm}/\SKc)\times_{(\mathrm{Sch}_{\mathrm{flat},\mathrm{lfp}}/\Spec(k))} (\mathrm{Sch}_{\mathrm{flat},\mathrm{lfp}}/\Spec(K^+/K^{\circ\circ})).
    \end{equation*}
    Moreover, it is compatible with 
    \begin{enumerate}[label=(\arabic*)]
    \item restricting to the following properties of morphisms on either side: étale, smooth, proper, finite, flat. 
    \item considering group spaces in all components.
    \end{enumerate}    
\end{proposition}
\begin{proof}
    Set $\cS:=\SKp$. Observe that for all $n$, 
    % https://q.uiver.app/#q=WzAsNCxbMSwxLCJcXFNwZWMoS14rLyhcXHZhcnBpIEteXFxjaXJjKV5uKT17XFxjU31fbiJdLFswLDEsIntcXGNTfV9uXlxcY2lyYz1cXFNwZWMoS15cXGNpcmMvXFx2YXJwaV5uKSJdLFswLDAsIlxcd2lkZXRpbGRle1xcY1N9XlxcY2lyYz1cXFNwZWMoaykiXSxbMSwwLCJcXFNwZWMoS14rL0tee1xcY2lyY1xcY2lyY30pPVxcd2lkZXRpbGRle1xcY1N9Il0sWzIsMV0sWzMsMF0sWzIsM10sWzEsMF0sWzAsNiwiIiwwLHsibGV2ZWwiOjEsInN0eWxlIjp7Im5hbWUiOiJjb3JuZXIifX1dXQ==
\[\begin{tikzcd}[row sep=1.1em]
	{\widetilde{\cS}^\circ=\Spec(k)} & {\Spec(K^+/K^{\circ\circ})=\widetilde{\cS}} \\
	{{\cS}_n^\circ=\Spec(K^\circ/\varpi^n)} & {\Spec(K^+/(\varpi K^\circ)^n)={\cS}_n}
	\arrow[""{name=0, anchor=center, inner sep=0}, from=1-1, to=1-2]
	\arrow[from=1-1, to=2-1]
	\arrow[from=1-2, to=2-2]
	\arrow[from=2-1, to=2-2]
	\arrow["\lrcorner"{anchor=center, pos=0.125, rotate=180}, draw=none, from=2-2, to=0]
\end{tikzcd}\]
 is a pushout square (cf.\ \spcite{0ET0}). Hence, by \spcite{07RV,07RX},
 \begin{equation}\label{eq:level_n_pushout}
    (\mathrm{Sch}_{\mathrm{flat},\mathrm{lfp}}/\cS_n)\isomarrow (\mathrm{Sch}_{\mathrm{flat},\mathrm{lfp}}/\cS_n^\circ)\times_{(\mathrm{Sch}_{\mathrm{flat},\mathrm{lfp}}/\widetilde{\cS}^\circ)} (\mathrm{Sch}_{\mathrm{flat},\mathrm{lfp}}/\widetilde{\cS})
    \end{equation}
    is an equivalence of categories for all $n$. It respects and reflects the properties in (1) by, e.g., \spcite{0CF4}, which applies to the thickenings. The statement (2) for \cref{eq:level_n_pushout} is immediate. The claim now follows from taking $2\text{-}\varprojlim\nolimits_n$ in \cref{eq:level_n_pushout}, see \cite[p.\ 289, I.2.1.8, I.2.3.2, I.4.8.1]{FujiwaraKato2018}.
\end{proof}

In particular, extending an admissible formal scheme or a morphism of admissible formal schemes over $\SKc$ to $\SKp$ is equivalent to extending its reduction to $\Spec(K^+/K^{\circ\circ})$. Let us make the following basic observations.

\begin{lemma}\label{l:generic_fiber_properties}
Let $f\colon \cX\to \cY$ be a morphism of admissible formal schemes over $\SKp$.
\begin{enumerate}
\item If $f$ is smooth, then $f\genf$ is smooth.
\item $f$ is proper if and only if $f\genf$ is proper.
\end{enumerate}
\end{lemma}
\begin{proof}
 (1)   This is known in the classical case $K^+=K^\circ$. The general case is a consequence of \cref{l:circ_reflects_tft_smooth} since $(-)\genf$ is compatible with $(-)\times_{\SKp}\SKc$. (2) is a special case of \cite[Prop.\ 3.12.5]{Huber1993Bewertungsspektrum}. The arguments do not need the Noetherian hypothesis.
\end{proof}

By algebraically compactifying the reduction of formal models over $\SKc$, \cref{p:pushout_equivalence} allows us to compactify certain spaces to $\SKKp$.

\begin{proposition}\label{p:equivalence_proper}
    Let $S$ be an adic space that is locally of topologically finite type over $\SKKp$.
    The functor
        $(-)^\circ\colon \mathrm{Adic}_{S}^{\mathrm{proper,tft}} \to \mathrm{Adic}_{S^\circ}^{\mathrm{proper,tft}}$
    is an equivalence, where $\mathrm{tft}$ is shorthand for 'topologically finite type'.
\end{proposition}
\begin{proof}
    By \cref{p:valuative_strong_criterion} the functor is fully faithful. Hence by gluing one may reduce to the case where $S$ is furthermore quasi-compact and quasi-separated. Let $f^\circ\colon X^\circ\to S^\circ$ be a proper morphism of topologically finite type. By \cite[Thm.\ 6.13]{Dine2025FormalModels} there exist quasi-compact and quasi-separated admissible formal models $\cX^\circ\to \SKc$ of $X^\circ$ and $\cS\to \SKp$ of $S$. Note that the model produced from \cite[Thm.\ 6.13]{Dine2025FormalModels} is indeed admissible in our sense by \cite[Cor.\ I.2.3.2]{FujiwaraKato2018}. After replacing $\cX^\circ$ by an admissible formal blowup, $f^\circ$ descends to a morphism $\overline{f}^\circ\colon \cX^\circ\to \cS^\circ$. It is proper by \cref{l:generic_fiber_properties}.
   
    Let $V:= K^+/K^{\circ\circ}$. By approximating $\Spec(k)\subseteq\Spec(V)$ by affine open neighborhoods and using \spcite{01ZL,081C,081F}, there exists a nonempty affine open $U\subseteq \Spec(V)$ such that $\cX_k$ spreads out to a flat scheme $\cX_U\to U$, and $\overline{f}^\circ_k$  spreads out to a proper morphism $\cX_U\to \cS_U$. By \spcite{0F41} there exists a proper scheme $\cX_{\Spec(V)}\to\cS_{\Spec(V)}$ and a schematically dense open immersion $j\colon\cX_U\to \cX_{\Spec(V)}$ over $\cS_{\Spec(V)}$. Properness of $\cX_U\to \cS_U$ shows that $j$ is an isomorphism over $U$. Since $V$ is a valuation ring, a scheme over $\Spec(V)$ is flat if and only if its structure sheaf is $V$-torsion free. Hence we can construct a $V$-flattening $\cX_{\Spec(V)}^\mathrm{fl}\xhookrightarrow{} \cX_{\Spec(V)}$ by considering the closed subscheme defined by the ideal sheaf of $V$-torsion elements. This ideal is trivial over $U$. By \spcite{05IN} $\cX^\mathrm{fl}_{\Spec(V)}\to \Spec(V)$ is of finite presentation, and $\smash{\cX^\mathrm{fl}_{\Spec(V)}\to \cS_{\Spec(V)}}$ remains proper as a composition of proper maps. It therefore gives rise to an admissible formal scheme $\cX\to \SKp$ and a proper morphism $\overline{f}\colon \cX\to \cS$ with $\SKc$-fiber $\overline{f}^\circ$ by \cref{p:pushout_equivalence}. Since $\cX\genf\to S$ is proper by \cref{l:generic_fiber_properties} this shows the equivalence.
\end{proof}

\section{\texorpdfstring{The uniformization cross over $\SKKp$}
    {The uniformization cross over Spa(K,K⁺)}}\label{s:4}
    
We now show the generalization of \cref{t:classical_uniformization_analytic} over $\SKKp$. 
There is a notable difference to the classical case, however: the abelian part of the uniformizing space $E$ may no longer have good reduction over the whole space $\SKKp$.
{Let us begin with the following obvious definition.}

\begin{definition}
    An \textit{abeloid variety} over $\SKKp$ is a smooth, proper group space $A\to \SKKp$ with geometrically connected fibers.
\end{definition}

We may replace \textit{geometrically connected fibers} by \textit{connected} in the above definition.
Indeed, it is a consequence of, e.g., \cref{l:dense_cocont}, that connectedness can be checked on the $\SKKc$-fiber. Therefore having geometrically connected fibers is also a property of the $\SKKc$-fiber. As in the case of schemes, a quasi-separated and connected adic space locally of topologically finite type over $\SKKc$ that admits a $\SKKc$-rational point is geometrically connected, see \cite[Thm.\ 3.2.1]{Conrad1999IrreducibleComponents} and also \cite[Lem.\ 3.4]{Mann2023NormalIrreducibleAdicSpaces}.

The adic space $A^{\mathrm{an}}$ over $\SKKp$ associated to an abelian variety $A\to\Spec(K)$ is an abeloid variety. By adic GAGA, the abelian varieties are exactly the projective abeloid varieties.

\begin{lemma}\label{l:abeloidEquivalence} Taking the $K^\circ$-fiber defines an equivalence of categories between abeloid varieties over $\SKKp$ and abeloid varieties over $\SKKc$.
\end{lemma}
\begin{proof}
This follows from \cref{p:equivalence_proper,l:circ_reflects_tft_smooth}.
\end{proof}

In particular,  abeloid varieties are commutative because commutativity holds in the classical case $K^+=K^\circ$.

\begin{definition}\label{d:good_reduction}
    An abeloid variety $A\to \SKKp$ has \textit{good reduction} if it admits a formal group model $\cA\to \SKp$ which is smooth. It has \textit{generically good reduction} if $A^\circ\to \SKKc$ has good reduction.
\end{definition}

\begin{remark}\label{r:good_reduction_abelian}
    Every smooth formal group model of an abeloid variety $A\to \SKKp$ is a formal abelian scheme.
\end{remark}
\begin{proof}
    Suppose that $\cA\to \SKp$ is a smooth formal group model of $A$. Then $\cA\to \SKp$ is proper by \cref{l:generic_fiber_properties}. Moreover, $\cA$ is connected since $A$ is connected. By considering the reduction over $\Spec(K^+/K^{\circ\circ})$,
    the statement follows from
    Stein factorization which implies that 
    a smooth, proper morphism of connected schemes that admits a section already has geometrically connected fibers, see, e.g., \cite[Prop.\ 27.191]{GoertzWedhorn2023AlgebraicGeometryII}.
\end{proof}

\begin{definition}\label{d:lattice}
    Let $T\to \SKKp$ be a split algebraic torus and let $B\to \SKKp$ be an abeloid variety with generically good reduction. Let $0\to T\to E\to B\to 0$ be an extension. 
    A \textit{(full) lattice} $M\xhookrightarrow{} E$ is a closed subgroup which is isomorphic to the constant group space $\underline{\dZ^r}:=\coprod_{\dZ^r} \SKKp$, where $r$ is the torus rank of $T$. A morphism of lattices is a morphism over $E$.
\end{definition}

Note that in the case $E \cong \dG_{m,\,\SKKp}$ lattices correspond to subgroups generated by elements $q \in K^\times$ with $K^\circ$-valuation of $q$ strictly smaller than $1$.
The definition agrees with the classical definition \cite[p.\ 257]{Luetkebohmert2016RigidGeometryCurves} when $K^+=K^\circ$.

\begin{lemma}\label{l:Ext_abeloid_circ_iso}
    Let $A\to \SKKp$ be an abeloid variety. Then the base change map $\mathrm{Ext^1_{an}}(A, \dG_{m,\SKKp})\to \mathrm{Ext^1_{an}}(A^\circ, \dG_{m,\SKKc})$ is an isomorphism.
\end{lemma}
\begin{proof}
    By the proof of \cite[Thm.\ 3.1]{RibeiroRosengarten2025Extensions}, $\mathrm{Mor}_e(A^i,\dG_{m})=\mathrm{Mor}_e((A^\circ)^i,\dG_{m})=0$ for $i=1,2,3$ implies that
    \begin{equation*}
        \mathrm{Ext^1_{an}}(A,\dG_{m,\SKKp})\cong H^1_m(A,\dG_{m,A})\cong H^1_m(A^\circ,\dG_{m,A^\circ})\cong \mathrm{Ext^1_{an}}(A^\circ,\dG_{m,\SKKc}),
    \end{equation*}
    where the middle isomorphism holds by \cref{r:cohomology_comparison}.
\end{proof}

Note that every extension of representable sheaves on the analytic site is itself representable by descent. Similarly, every extension of an abeloid variety by a partially proper group space is itself partially proper over $\SKKp$.

\begin{theorem}\label{t:uniformization_analytic}
   Let $A$ be an abeloid variety over $\SKKp$.
   After a finite separable affinoid field extension,
   there exists an isomorphism $A\cong E/M$, where $0 \to T\to E\to B \to 0$ is an extension of an abeloid variety $B$ with generically good reduction by a split algebraic torus $T\cong \dG_{m,\SKKp}^r$, and $\underline{\dZ^r}\cong M\subseteq{} E$ is a full lattice. 
   
   In other words, we have a diagram of short exact sequences of adic group spaces over $\SKKp$, which we call the \textit{uniformization cross} of $A\cong E/M$:
        % https://q.uiver.app/#q=WzAsOSxbMSwyLCJUIl0sWzIsMiwiRSJdLFszLDIsIkIiXSxbMCwyLCIwIl0sWzQsMiwiMCJdLFsyLDEsIk0iXSxbMiwzLCJBIl0sWzIsMCwiMCJdLFsyLDQsIjAiXSxbMCwxXSxbMSwyXSxbMywwXSxbMiw0XSxbMSw2XSxbNSwxXSxbNyw1XSxbNiw4XV0=
\[\begin{tikzcd}[column sep=0.9em,row sep=0.7em]
	&& 0 && \\
	&& M \\
	0 & T & E & B & 0. \\
	&& A \\
	&& 0
	\arrow[from=1-3, to=2-3]
	\arrow[from=2-3, to=3-3]
	\arrow[from=3-1, to=3-2]
	\arrow[from=3-2, to=3-3]
	\arrow[from=3-3, to=3-4]
	\arrow[from=3-3, to=4-3]
	\arrow[from=3-4, to=3-5]
	\arrow[from=4-3, to=5-3]
\end{tikzcd}\]
The horizontal extension is called the Raynaud extension of $A$. Moreover,
\begin{enumerate}
    \item Every {homomorphism}
    between abeloid varieties lifts uniquely to a morphism between their uniformization crosses. In particular, the uniformization cross is unique up to unique isomorphism lifting $\mathrm{id}_A$.
    \item $E\to A$ is the universal (analytic) covering space of $A$ in the sense that $H^1(E_{\mathrm{an}}, \dZ)=0$ and that $E\to A$ is final among pointed morphisms $S\to A$ where $S$ is a connected adic space with $H^1(S_{\mathrm{an}}, \dZ)=0$.
    \item The uniformization cross for $A$ exists if and only if the uniformization cross exists for $A^\circ$. In this case, the cross of $A^\circ$ is the $\SKKc$-fiber of the cross of $A$.
\end{enumerate} 
\end{theorem}
\begin{proof}
    Let $K'/K$ be a finite separable field extension that makes the conclusion of \cref{t:classical_uniformization_analytic} hold for $A^\circ\times_{\SKKc} \Spa(K', K'^\circ)$. By \spcite{00IA} there exists a valuation ring $V'$ of the residue field of $K'^\circ$ dominating $K^+/{K^{\circ\circ}}$. Then the preimage of $V'$ under the reduction map $K'^\circ\to K'^\circ/K'^{\circ\circ}$
    is a valuation ring of $K'$ dominating $K^+$. It is open as it contains $K'^{\circ\circ}$, and therefore defines an affinoid field $S$. We base-change $A$ to $S$ and replace $\SKKp$ by $S$ to reduce notation.
    
    Denote the objects appearing in the uniformization cross of $A^\circ$ by the same letters together with the symbol $(-)^\circ$. Note that $B^\circ$ uniquely lifts to an abeloid variety $B\to \SKKp$ by \cref{l:abeloidEquivalence}. It is clear that $M^\circ$ and $T^\circ$ lie in the essential image of the fully faithful functor  $(-)^\circ\colon \mathrm{Adic}_{\SKKp}^{\mathrm{p.\ proper}} \to \mathrm{Adic}_{\SKKc}^{\mathrm{p.\ proper}}$. By \cref{l:Ext_abeloid_circ_iso}, $E^\circ$ also lies in the essential image. Hence the uniformization cross over $\SKKc$ uniquely lifts to a diagram of group spaces over $\SKKp$. They form short exact sequences by \cref{l:circ_reflect_pp}. (1) follows from the full faithfulness of $(-)^\circ$ for partially proper spaces.
    (2) follows from \cref{p:valuative_strong_criterion,r:cohomology_comparison,t:classical_uniformization_analytic}. (3) is immediate from the construction.
\end{proof}

In view of this result one might speculate by analogy with the classical case that the abelian part $B$ has a good reduction model over the ring of integers $K^+$. However, this is not always the case, see \cref{e:classification_RU_curves}. \Cref{c:uniformization_formal3} describes the resulting obstruction in terms of the marked Néron model of $A$.

\begin{definition}\label{d:dual_abeloid}
    The translation-invariant Picard functor $\mathrm{Pic}_A^{\tau}$ associated to an abeloid variety $A\to \SKKp$ is the presheaf of groups mapping an adic space $T$ locally of topologically finite type over $\SKKp$ to 
\begin{equation*}
H^1_m(A_T,\dG_m):=\mathrm{ker}\Big(H^1(A_T,\dG_m)\xrightarrow{m^*-\mathrm{pr}_1^*-\mathrm{pr}_2^*} H^1((A\times A)_T,\dG_m)\Big).
\end{equation*}
    The group space representing this presheaf is called the dual, denoted by $A^\vee$, if it exists.
\end{definition}

By \cite[Thm.\ 3.1]{RibeiroRosengarten2025Extensions}, $\mathrm{Pic}_A^{\tau}$ is a sheaf on the étale site, and there is a natural isomorphism
\begin{equation}\label{eq:dual_ext_picard}
    \mathrm{Pic}_A^{\tau}(T)\cong \mathrm{Ext}^1\hspace{-0.2mm}(A_T, \dG_{m,T}).
\end{equation}
Note that every extension in the right-hand side comes equipped with a canonical rigidification along $e_{A_T}$.
By \cite[Cor.\ 7.6.5]{Luetkebohmert2016RigidGeometryCurves}, the dual is representable in the classical case $K^+=K^\circ$.

The next result complements the discussion of abeloid varieties over $\SKKp$ in \cite[section 4.3]{Gerth2026}. Note that Raynaud uniformizations in loc.\ cit.\ require the abelian quotient to have good reduction, which differs from our setting.

\begin{corollary}\label{c:dual_exists}
    Let $A\to \SKKp$ be an abeloid variety. The dual $A^\vee\to \SKKp$ exists and is an abeloid variety.
\end{corollary}
\begin{proof}
    {Let $((A^\circ)^\vee)^{c}$ be the unique extension of $(A^\circ)^\vee$ to an abeloid variety over $\SKKp$ by \cref{l:abeloidEquivalence}}. It is a consequence of \cref{l:circ_embedding} that $\mathrm{Pic}(X)\cong \mathrm{Pic}(X^\circ)$ are naturally isomorphic for every adic space $X\to \SKKp$. In particular, $\mathrm{Pic}_A^{\tau}$ is the pushforward of $\mathrm{Pic}_{A^\circ}^{\tau}$ along $\SKKc\to \SKKp$, and is therefore represented by $((A^\circ)^\vee)^{c}$ by \cref{p:valuative_strong_criterion}.
\end{proof}

One can lift the relation between the uniformization cross of $A^\vee$ and that of $A$ from the $\SKKc$-fiber \cite[Prop.\ 6.1.9, Thm 6.3.3, Prop.\ 6.4.1]{Luetkebohmert2016RigidGeometryCurves} to $\SKKp$ by using \cref{l:circ_embedding,l:Ext_abeloid_circ_iso,p:valuative_strong_criterion}.

\section{N\'eron models over general valuation rings}\label{s:5}

We use the term Néron model without requiring quasi-compactness over the base. Our definition of Néron model thus agrees with the notion of a Néron lft-model in \cite{BoschLuetkebohmertRaynaud1990NeronModels}. Such a theory of N\'eron lft-models was initiated in \cite[Chapter 10]{BoschLuetkebohmertRaynaud1990NeronModels}.

\begin{definition}\label{d:Neron_model_V}
	Let $S$ be a normal integral scheme with generic point $\eta$. Let $G_\eta\to S_\eta$ be a commutative group scheme.
    \begin{enumerate}[leftmargin=3em]
	    \item A Néron model for $G_\eta$ is a smooth separated scheme $G\to S$ with generic fiber $G_\eta$ satisfying the Néron mapping property $G(X)\cong G_\eta(X_\eta)$ for all smooth
        $S$-schemes $X$. 
        \item A marked N\'eron model for $G_\eta$ is a smooth separated scheme $G\to S$ with connected fibers and generic fiber $G_\eta$, together with a section extending the unit section of $G_\eta$, satisfying the following universal property: for every smooth, 
        quasi-separated scheme with connected fibers $X\to S$
        together with a section $S\to X$, the base change map between sets of pointed morphisms $\mathrm{Mor}^\mathrm{pt}(X,G)\to \mathrm{Mor}^\mathrm{pt}(X_\eta, G_\eta)$ is an isomorphism (where pointed means that the fixed section is mapped to the unit section).
	\end{enumerate}
\end{definition}

A (marked) Néron model is unique {up to unique isomorphism} if it exists. In both cases the mapping property induces a group structure on $G$, making it a group model of $G_\eta$. Moreover, observe that if $G_\eta$ {is connected and} admits a Néron model, then its fiberwise identity component is the marked Néron model.

Our first task is to show that semi-abelian models are in fact marked Néron models. Recall that a smooth commutative group scheme $E\to S$ is \textit{semi-abelian} if all fibers over $S$ are extensions of abelian varieties by  tori. 

We often use the fact that a morphism of schemes with connected fibers which admits a section has geometrically connected fibers \spcite{04KV}. In particular, semi-abelian schemes have geometrically connected fibers.  Moreover, every semi-abelian scheme is separated and quasi-compact over the base, since it is smooth with geometrically connected fibers (see, e.g.,  \cite[Exp.\ VI$_B$, Cor.\ 5.5]{DemazureGrothendieck1970SGA3IVB}).
In order to descend semi-abelian schemes we  need the following lemma.

\begin{lemma}\label{l:semi-abelian_approximation}
	Let $S$ be a quasi-compact and quasi-separated scheme and let $E\to S$ be a semi-abelian scheme. Let $(S_i, f_{ij})$ be an inverse system over a directed index set $I$ of quasi-compact and quasi-separated schemes with affine transition maps and $\varprojlim_{i \in I} S_i=S$. Then for all sufficiently large $j\in I$ there exists a semi-abelian scheme $E_j\to S_j$ such that $E_j\times_{S_j} S=E$.
\end{lemma}
\begin{proof}
    Let us first consider an inverse system
     $(S'_{i}, f'_{ij})$ 
     over a directed index set $I'$ 
     consisting
     of finite type, quasi-compact quasi-separated schemes over $\Spec(\dZ)$ with affine transition maps and $\varprojlim_{I'} S'_i=S$. Such an inverse system exists by \spcite{01ZA}.
    
    By \spcite{01ZL,081C} there exists $i_0\in I'$ such that $E$ descends to a smooth separated commutative group scheme $E_{i_0}\to S'_{i_0}$. For each $i\geq i_0$ let $M_i\subseteq S'_{i}$ be the subset of $s\in S'_{i}$ such that $E_{{i},s}\to \Spec(\kappa(s))$ is semi-abelian.
    
    We claim that this is a union of constructible subsets of $S'_{i}$. Indeed, if $s\in M_i$, then the extension over $\kappa(s)$ spreads out over an open subspace of $\overline{\{s\}}_{\mathrm{red}}$. To see this, let $0\to T_{s}\to E_s\to A_s\to 0$ be an extension over $\kappa(s)=\cO_{\overline{\{s\}}_{\mathrm{red}},s}$ where $T_s$ is a torus and $A_s$ is an abelian variety. Let $s'\to s$ be an fppf cover that splits $T_s$. Writing $\Spec(\cO_{\overline{\{s\}}_{\mathrm{red}},s})=\varprojlim_{U\subseteq \overline{\{s\}}_{\mathrm{red}}} U$ ranging over nonempty affine opens, it follows from \spcite{05FI,081F} that the extension, the fppf-map, and the splitting descend for a sufficiently small $U$. This shows the claim.
    
    Moreover, $f_{ii'}'^{-1}(M_{i'})=M_i$ for $i\geq i'\geq i_0$. Similarly, the image of $S$ is contained in $M_i$. Hence the statement
    for the system $(S'_{i}, f'_{ij})$ follows from \cite[IV$_3$, Cor.\ 8.3.4]{GrothendieckDieudonne1966EGAIV}.

    To deduce the lemma for {a general} system $(S_i, f_{ij})$, let $(S'_{i}, f'_{ij})$ be as above. Let $j_0\in I'$ be such that $E_{j_0}\to S_{j_0}'$ is semi-abelian. Note that $\Mor(S, S'_{j_0})=\varinjlim_{i\in I}\Mor(S_i, S_{j_0}')$ by \spcite{01ZC} as $S_{j_0}'$ is finitely presented. Hence there exists an $i_0\in I$ and a factorization $S\to S_{i_0}\to S'_{j_0}$. The statement follows by forming the pullback of the semi-abelian scheme over $S'_{j_0}$ to $S_{i_0}$.
\end{proof}

We  will often use the fact that over every integral scheme $S$ with generic point $\eta$, the map $\mathrm{Mor}(X,Y)\to \mathrm{Mor}(X_\eta, Y_\eta)$ is injective whenever
$X$ is flat over $S$ and $Y$ is separated over $S$. 

\begin{proposition}\label{p:semi-abelianAreNeron}
    Let $S$ be a normal integral scheme with generic point $\eta$. Let $E\to S$ be a semi-abelian scheme. Then $E$ is the marked Néron model of $E_\eta$.
\end{proposition}
\begin{proof}
    Suppose that $X$ is a smooth,
    quasi-separated
    $S$-scheme with connected fibers together with a section
    $x\in X(S)$. We need to show that $\mathrm{Mor}^\mathrm{pt}(X,E)\to \mathrm{Mor}^\mathrm{pt}(X_\eta, E_\eta)$ is bijective. Since this statement is local on $S$, we may assume that $S=\Spec(R)$ is affine. The injectivity of the map follows from the flatness of $X$ and the separatedness of $E$.
    
    To show surjectivity, let a pointed morphism $f_\eta\colon X_\eta\to E_\eta$ and $x_0\in X$ be given. We will show that $f_\eta$ can be extended to a neighborhood of $x_0$. Denote the image of $x_0$ by $s_0\in S$. Let $W\subseteq X$ be a quasi-compact neighborhood of the image $x(s_0)$ of our section $x$ and of $x_0$. Then $W$ is quasi-separated as $X$ is quasi-separated. Note that $x$ restricts to a section $x'\colon \Spec(\cO_{S,s_0})\to {W} \times_S \Spec(\cO_{S,s_0})=: X'$. Moreover, $X'$ has connected fibers since $X$ has integral fibers.
    
    We claim that the restriction of $f_\eta$ to $X'_\eta$ extends to a morphism $X'\to E\times_S \Spec(\cO_{S,s_0})=:E'$. This implies our statement since we can approximate $\Spec(\cO_{S,s_0})$ by affine open neighborhoods of $s_0$ and use \spcite{01ZL}. Since $\cO_{S,s_0}$ is the filtered union of Noetherian normal subrings, $X',x'$ and $E'$ descend with good properties (by, e.g., \cref{l:semi-abelian_approximation}, \spcite{01ZL,081C,05FI}) to a normal integral Noetherian base $S_0$ with generic point $\eta_0$. We denote the descended objects with a subindex $(-)_0$. The generic morphism $f_\eta$ may be assumed to descend as a pointed morphism $f_{\eta_0}\colon X'_{0,\eta_0}\to E'_{0,\eta_0}$.
   
    Let $s\in S_0$ be a point of codimension $1$. We show that there exists an open $U \subseteq S_0$ with $s\in U$ such that $f_{\eta_0}$ extends to a morphism $X'_{0,U}\to E'_{0,U}$. Let $T'_{0,\eta_0}$ be the toric part of $E'_{0,\eta_0}$, and let $A'_{0,\eta_0}$ be its abelian part. By \cite[Prop.\ I.2.9]{FaltingsChai1990Degeneration}, $T'_{0,\eta_0}$ extends to a closed subtorus $T'_0\to E'_0$. Over $\Spec({\cO_{S_0,s}})$, the quotient $A'_{0,\cO_{S_0,s}}:=E'_{0,\cO_{S_0,s}}/T'_{0,\cO_{S_0,s}}$ is a semi-abelian scheme. Hence, by \cite[Prop.\ 7.4.3]{BoschLuetkebohmertRaynaud1990NeronModels}, it is the fiberwise identity component of the Néron model of $A'_{0,\eta_0}$ over $\Spec({\cO_{S_0,s}})$. Therefore, as seen in the proof of \cite[Prop.\ 10.1.7]{BoschLuetkebohmertRaynaud1990NeronModels}, $E'_{0,\eta_0}$ admits a  {(lft-)}Néron model over $\Spec({\cO_{S_0,s}})$ whose fiberwise identity component is $E'_{0,\cO_{S_0,s}}$. In particular, $f_{\eta_0}$ extends to a morphism $X'_{0,\cO_{S_0,s}}\to E'_{0,\cO_{S_0,s}}$. Hence we may assume that $U$ contains all codimension $1$ points of $S_0$.
    
    By \cite[IX 1.1]{Raynaud1970FaisceauxAmples}, the rational map $f_0\colon X'_0\dashrightarrow E'_0$ is defined in a neighborhood of the section $S_0\to X'_0$, as depth one points in the normal Noetherian scheme $S_0$ have codimension one. Therefore, since $X'_{0,U}$ contains the codimension $\leq 1$ points of $X'_0$ by flatness, Weil's extension theorem \cite[Thm.\ 4.4.1]{BoschLuetkebohmertRaynaud1990NeronModels} provides an extension $X'_0\to E'_0$. This shows the claim.
\end{proof}

\begin{remark}
    The preceding proposition can be viewed as a generalization of \cite[Prop.\ I.2.7]{FaltingsChai1990Degeneration} which states that rational homomorphisms between two semi-abelian schemes over a Noetherian normal base are defined everywhere.
\end{remark}

\begin{lemma}\label{l:semi-abelian_ext}
	Let $S:=\Spec(V)$ where $V$ is either a valuation ring with separably closed field of fractions or a regular local ring. Denote the generic point by $\eta$. Suppose that $E\to S$ is a semi-abelian scheme whose generic fiber is an abelian variety. Then the natural map
    \begin{equation*}
        \mathrm{Ext}^1\hspace{-0.2mm}(E, \dG_{m,S})\isomlongarrow \mathrm{Ext}^1\hspace{-0.2mm}(E_\eta, \dG_{m,S_\eta})
    \end{equation*}
    is an isomorphism.
\end{lemma}
\begin{proof}
	By \cite[Prop.\ 4.2.1, Cor.\ 4.2.4]{AntieauDatta2021ValuationRingsDerivedSplinters}, $V$ is a filtered colimit of regular subrings $V_i$. By localizing each $V_i$ at the preimage of the maximal ideal in $V$, we may further assume that each $V_i$ is local.
   Together with
    \cref{l:semi-abelian_approximation} and \spcite{01ZL,081F}, it is straightforward to reduce the proof to the case where $V$ is a regular local ring.

   Since $\mathrm{Pic}(S)=0$, we find that $\mathrm{Pic}(E)\to \mathrm{Pic}(E_{\eta})$ is injective by \spcite{0BD7}. On the other hand, the regularity of $S$ shows that the map is surjective, see \spcite{0BD9,0B8W}. The same argument applies to $E^2= E\times_S E$. 
    It follows that $H^1_m(E, \dG_m)\to H^1_m(E_{\eta}, \dG_m)$ is an isomorphism, where $H^1_m$ is defined as in \cref{d:dual_abeloid}.
    
    The natural maps $\mathrm{Ext}^1\hspace{-0.2mm}(E, \dG_{m,S})\to H^1_{m}(E, \dG_m)$ and $\mathrm{Ext}^1\hspace{-0.2mm}(E_\eta, \dG_{m,S_\eta})\to H^1_{m}(E_\eta, \dG_{m})$ are isomorphisms. Indeed, by the proof of \cite[Thm.\ 3.1]{RibeiroRosengarten2025Extensions}, it suffices to verify that $\mathrm{Mor}_{e}(E^i,\dG_{m,S})= \mathrm{Mor}_{e}(E^i_{\eta},\dG_{m,\eta})=0$ for $i=1,2,3$. This is true as $E_\eta $ is an abelian variety. The statement follows.
\end{proof}

\begin{lemma}\label{l:semi-abelian-model}
    Let $V$ be a valuation ring with fraction field $F$, and let $E_F\to \Spec(F)$ be a semi-abelian variety. After replacing $F$ by a finite separable field extension and $V$ by a suitable choice of a dominating valuation ring, $E_F$ extends to a semi-abelian scheme over $\Spec(V)$.
\end{lemma}
\begin{proof}
    We first consider the case where $F$ is separably closed. Let $T_F\to E_F$ be its maximal torus, and let $A_F:=E_F/T_F$ be the abelian part. Let $\overline{\cA_g}$ be a toroidal compactification of the moduli space of principally polarized abelian varieties $\cA_g$, see \cite[Thm.\ IV.5.7]{FaltingsChai1990Degeneration}. It is a proper DM-stack over $\Spec(\dZ)$, and the universal abelian scheme over $\cA_g$ extends to a semi-abelian scheme over $\overline{\cA_g}$. 
    
    By Zarhin's trick, $A_F^4\times_F( A_F^\vee)^4$ is principally polarized, and hence corresponds to a point in $\cA_g(F)$. 
    The valuative criterion \cite[Thm.\ 7.3]{LaumonMoretBailly2000ChampsAlgebriques} implies that $A_F^4\times_F( A_F^\vee)^4$ extends to a semi-abelian scheme $H\to \Spec(V)$. Denote by $\pi_F\colon A_F^4\times_F( A_F^\vee)^4\to A_F^4\times_F( A_F^\vee)^4$ the idempotent projection onto the last $7$ of its $8$ factors. By \cref{p:semi-abelianAreNeron}, $\pi_F$ lifts uniquely to an idempotent homomorphism $\pi\colon H\to H$. Then $H\cong \mathrm{ker}(\pi)\times_{\Spec(V)} \mathrm{ker}(1-\pi)$, see, e.g., \spcite{09SF}. Observe that $\mathrm{ker(\pi)}$ and $\mathrm{ker}(1-\pi)$ are commutative locally finite presentation group schemes with geometrically connected fibers. Moreover, they are smooth over $\Spec(V)$, since a retract of a formally smooth $\Spec(V)$-scheme is formally smooth. It follows that $\mathrm{ker(\pi)}$ and $\mathrm{ker}(1-\pi)$ are semi-abelian schemes. Indeed, by Chevalley's theorem 
    each geometric fiber of either factor is uniquely an extension of an abelian variety by a smooth connected affine group variety. Since their product is semi-abelian, it follows that these affine contributions in both factors are tori. Hence $\mathrm{ker}(\pi)=: A'$ is a semi-abelian model of $A_F$. By \cref{l:semi-abelian_ext}, the extension given by the semi-abelian variety $E_F$ uniquely spreads out to an extension $0\to T\to E\to A'\to 0$.

    For general $V$, let $F^\mathrm{sep}$ be a separable closure of $F$, and choose a valuation ring $V^{\mathrm{sep}}$ of $F^\mathrm{sep}$ dominating $V$. Then $V^\mathrm{sep}$ is the filtered union of dominating valuation subrings $V_i\supseteq V$ whose fraction field extensions $\mathrm{Frac}(V_i)/F$ is finite. Indeed, every finite extension $F_i/F$ defines such $V_i$ by $V_i:=F_i\cap V^\mathrm{sep}$. Hence the statement follows from \spcite{01ZL} and \cref{l:semi-abelian_approximation}.
\end{proof}

\begin{theorem}\label{t:Neron_existence_V}
	Let $V$ be a Henselian valuation ring with fraction field $F$, and suppose that the residue fields of $\Spec(V)$ at non-generic points are algebraically closed. Let $E_F\to \Spec(F)$ be a semi-abelian variety. After replacing $F$ by a finite separable field extension and $V$ by the unique valuation ring dominating it,
    $E_F$ extends to a semi-abelian scheme over $\Spec(V)$.
    Once it does,
    $E_F$ admits a Néron model over $\Spec(V)$ whose fiberwise identity component is a semi-abelian scheme.
\end{theorem}
\begin{proof}
    Recall the characterization that a valuation ring is Henselian if and only if for every finite extension of the fraction field there is exactly one choice of a valuation ring dominating it. Set $S:=\Spec(V)$, $F:=\mathrm{Frac(V)}$, and let $E_F\to \Spec(F)$ be a semi-abelian
    variety. After replacing $F$ by a finite separable field extension and $V$ by a dominating valuation ring, $E_F$ admits a semi-abelian scheme model $E\to S$ by \cref{l:semi-abelian-model}. Note that this new $V$ again satisfies the assumptions of the theorem.
    
    The following construction follows that of the Néron model of the multiplicative group in \cite[Example 10.1.5]{BoschLuetkebohmertRaynaud1990NeronModels}. For each $g\in E_F(F)$, let $g^e$ be the extension of $g$ to its domain of definition $\mathrm{dom}(g)$ associated to the rational map
    $g\colon S\dashrightarrow E$.  We define
	\begin{equation*}
		G:= \bigcup_{g\in E_F(F)}g^e + E
	\end{equation*}
	as the result of gluing $E_F(F)$ copies of $E$, where the copy $g^e + E$ is glued to $h^e+E$ along $\mathrm{dom}(h-g)$ by the translation by $(h-g)^e$ on $\restrsmall{E}{\mathrm{dom}(h-g)}$. This defines a gluing datum. 
    The space $G$ is a smooth $S$-scheme with generic fiber $E_F$ and fiberwise identity component $E$. Moreover, it carries a group structure since the maps $(g^e + E)\times_S (f^e+E)\to (g+f)^e + E$ given by addition on $E$ glue.
    
    We claim that $G$ is separated. For a group scheme, this is equivalent to the unit section being closed. Thus, this reduces to the claim that for each $g\in G(F)$, $g^e\colon \mathrm{dom}(g^e)\to E$ is a closed immersion. Let $\Spec(R)\subseteq E$ be an affine open subspace. If $\Spec(R)$ does not contain $g$, then it does not intersect the image of $g^e$. Otherwise, $g$ corresponds to a factorization $V\to R\to F$. The image of $R$ in $F$ is a subring of $F$ containing $V$, and hence a local ring of $S$. Denote the corresponding point of $S$ by $s_R$. Then $\cO_{S,s_R}$ is the minimal local ring of $S$ over which $g$ extends to $\Spec(R)$. Note that this implies that $\Spec(\cO_{S,s_R})\subseteq S$ is open. Indeed, the image of $\Spec(R)$ in $S$ is open by flatness, and a map defined on a local ring spreads out to an open neighborhood since $\Spec(R)$ is of finite presentation over $S$. 
    This implies $(g^e)^{-1}(\Spec(R))=\Spec(\cO_{S,s_R})$. Surjectivity of $R\to \cO_{S,s_R}$ shows that the restriction $g^e\colon (g^e)^{-1}(\Spec(R))\to \Spec(R)$ is a closed immersion, which shows the claim.
	
	We now verify the Néron mapping property for $G$. Suppose that $X\to S$ is a smooth scheme. Injectivity of $G(X)\to G(X_F)$ follows as $G$ is separated and $X$ is flat. For surjectivity let $f_F\colon X_F\to E_F$ and $x_0\in X$ be given. We claim that $f_F$ extends to an open neighborhood of $x_0$. First, we may replace $X$ by an affine neighborhood of $x_0$. Let $s_0\in S$ be the image of $x_0$ in $S$. When $s_0$ is the generic point, the claim follows from approximating $s_0$ by affine open neighborhoods. Otherwise, denote by $y_0$ a closed point in the fiber $X_{s_0}$ generalizing to $x_0$.

    Since $\Spec(\cO_{S,s_0})$ is Henselian, $\kappa(s_0)$ is algebraically closed, and $X$ is smooth, $y_0$ extends to a section of $X_{\cO_{S,s_0}}$ over $\Spec(\cO_{S,s_0})$. This spreads out over an open subset $U\subseteq{S}$ to a section $y\colon U\to X_U$,  since we can approximate $\Spec(\cO_{S,s_0})$ by affine open neighborhoods of $s_0$ and use \spcite{01ZL}. The fiberwise connected component $X^0_U$ meeting this section is open in $X_U$, see \spcite{055R}. By construction, for every $g\in G(F)$, the unit section of 
    $g^e+E$
    gives rise to a section of $G$
    over $S$ extending $g$. Thus 
    $G(U)=G(F)$ by separatedness. Therefore $\Spec(F)\xrightarrow{y_F} X^0_F\to G_F$ extends uniquely to a section $\sigma\colon U\to G_U$. After translation, we may assume that $\sigma=e_{G_U}$. Thus $X^0_F\to G_F$ extends to a morphism $X^0_U\to E_U\subseteq G$ by
    \cref{p:semi-abelianAreNeron}.
\end{proof}

The assumptions on $V$ in the previous Theorem are satisfied in the following cases:
\begin{example}\label{e:Neron_excistence}
Let $V$ be a valuation ring with fraction field $F$. Then $V$ is Henselian and the residue fields at non-generic points of $\Spec(V)$ are algebraically closed in each of the following cases:
    \begin{enumerate}
        \item $F$ is algebraically closed.
        \item $(F,V)$ is an affinoid field whose residue field $F^\circ/F^{\circ\circ}$ is algebraically closed. Equivalently, $V$ has a prime ideal of height one, the residue field at that prime is algebraically closed, and $V$ is complete with respect to the valuation topology.
    \end{enumerate}
\end{example}
\begin{proof}
    Recall that $V$ is Henselian if and only if for every finite field extension $F'/F$ there is exactly one choice of a valuation ring $V'$ of $F'$ dominating $V$. This implies that (1) is an example. If $V$ satisfies (2), note that all residue fields are algebraically closed as $V/F^{\circ\circ}$
    is a valuation ring with algebraically closed fraction field $F^\circ/F^{\circ\circ}$. Let $F'/F$ be a finite field extension, and let $V'$ be a valuation ring of $F'$ dominating $V$. Then $(F', V')$ is an affinoid field and $F'^\circ\subseteq F'$ is unique as $F^\circ$ is complete. Now, the residue field extension $(F'^\circ/F'^{\circ\circ})/(F^{\circ}/F^{\circ\circ})$ is trivial by the assumption that $F^{\circ}/F^{\circ\circ}$ is algebraically closed. As $V'$ corresponds exactly to a valuation ring of $F'^\circ/F'^{\circ\circ}$ dominating $V/F^{\circ\circ}$, this shows the uniqueness of $V'$. Thus $V$ is Henselian.
\end{proof}

We can deduce the following results.

\begin{corollary}\label{c:Neron_existence_microbial}
    Suppose that the residue field $k$ is algebraically closed. Let $A\to \Spec(K)$ be an abelian variety. After a finite separable affinoid field extension,
    $A$ admits a Néron model over $\Spec(K^+)$.
\end{corollary}

\begin{corollary}\label{c:abelian_Neron}
    Let $V$ be a valuation ring with fraction field $F$.
    \begin{enumerate}[label=(\arabic*), leftmargin=2.6em]
        \item If $F$ is algebraically closed, then every abelian variety over $F$ has a Néron model over $V$.
        \item If $A\to \Spec(V)$ is an abelian scheme, then $A$ is the Néron model of its generic fiber.
    \end{enumerate}
\end{corollary}
\begin{proof}
(1) follows from \cref{t:Neron_existence_V}. (2): Choose an algebraic closure $F^{\mathrm{alg}}/F$ and a valuation ring $V^{\mathrm{alg}}$ of $F^{\mathrm{alg}}$ dominating $V$. By fpqc-descent, to show that a separated scheme over $V$ has the Néron mapping property for its generic fiber, it suffices to show it after base change to $V^{\mathrm{alg}}$. Hence we may assume that $V=V^{\mathrm{alg}}$. The Néron model $\cN_A$ of its generic fiber $A_F$ satisfies $\cN_A^0=A$ by the uniqueness of marked N\'eron models. As $A$ is proper, $A\subseteq \cN_A$ is also closed. Hence $\cN_A=A$ since $\cN_A$ is connected.
\end{proof}

Note that these results gives a partial answer to the question posted in \cite{mathoverflow-Neron} if abelian varieties have N\'eron models over arbitrary valuation rings. The second statement of \cref{c:abelian_Neron} 
was in fact proved with a different argument by K\k{e}stutis \v{C}esnavi\v{c}ius and David Holmes
in an answer to \cite{mathoverflow-Neron}.

\begin{lemma}\label{l:divisible_component_group}
	Let $V$ be a valuation ring with algebraically closed field of fractions $F$. Suppose that $E_F\to \Spec(F)$ is a semi-abelian variety with Néron model $G\to \Spec(V)$. Let $s\in \Spec(V)$. Then $\pi_0(G_s)$ is a divisible group.
	
	In particular, the Néron model is quasi-compact if and only if $E_F$ is an abelian variety with good reduction over $V$, or if $V=F$.
\end{lemma}
\begin{proof}
	For the first statement we may assume that $s$ is the closed point by considering the base change to $\Spec(\cO_{\Spec(V),s})$. Since every connected component of $G_s$ is hit by a section of $\Spec(V)$, 
    $\pi_0(G_s)\cong G(V)/G^0(V)=E_F(F)/G^0(V)$.
    It is divisible as $E_F$ is an extension of an abelian variety by a torus, and $F$ is algebraically closed.

	The if part of the last statement of the lemma is clear. For the other direction, suppose that $G$ is quasi-compact. Thus $\pi_0(G_s)$ is finite and hence trivial for every $s\in \Spec(V)$. Therefore $G=G^0$ and $G^0(V)=G^0(F)$. It follows from considering torsion points that the torus rank of $G^0$ is constant. Indeed, if $m$ is invertible in $V$, then $G[m]$ is étale, and hence
		\begin{equation*}
		m^{2g-t(E_F)}=\# G[m](F)=\# G[m](\cO_{\Spec(V),s})=\# G[m](\kappa(s))=m^{2g-t(G_s)},
		\end{equation*}
		where $g:=\mathrm{dim}(E_F)$, and $t(-)$ denotes the torus rank. Thus $G$ is globally an extension of an abelian scheme by a (split) torus (cf.\ \cite[Cor.\ I.2.11]{FaltingsChai1990Degeneration}). For $V\neq F$, the surjectivity of $G(V)\to G(F)$ implies that the torus rank is $0$.
\end{proof}

For example when $F=\dC_p$ and $E_F$ is a Tate curve with parameter $q\in \dC_p$, then the special fiber of the Néron model has component group $\dQ/v_p(q)\dZ$; see \cref{e:classification_RU_curves}.

We will see in \cref{p:Neron_finite_m}(2) that for every nonzero $m\in \dZ$, multiplication by $m$ is a finite and faithfully flat map on Néron models of
semi-abelian varieties over algebraically closed fields. This generally fails in the discretely valued setting.

\section{Formal N\'eron models over general valuation rings}\label{s:6}
In this section we return to the investigation of adic spaces over $\SKKp$ and their formal models. Recall our conventions \ref{convention_adic} and \ref{convention_formal} for this setting.
\subsection{Definitions and construction}
Our strategy for the construction of formal N\'eron models is to use the algebraic results of the previous section to lift models from $K^\circ$ to $K^+$.

\begin{definition}
    A \textit{formal semi-abelian (group) scheme} over $\SKp$ is a smooth commutative formal group scheme $\cE\to \SKp$ whose reduction $\cE\times_\SKp \Spec(K^+/K^{\circ\circ})\to \Spec(K^+/K^{\circ\circ})$ is a semi-abelian scheme.
\end{definition}

Formal semi-abelian schemes over $\SKp$ are separated and quasi-compact, as these properties can be tested on the reduction. 

Before giving the definition of a marked Néron model, let us show the following generalization of the classical theory of formal models which we will use freely.

\begin{lemma}\label{l:smooth_int_closed}
    Let $\cX\to \SKp$ be a smooth formal scheme. Then
    \begin{enumerate}
        \item the natural map $\cO_{\cX}\to (\sp_{\cX})_*\cO_{\cX\genf}^+$ is an isomorphism\footnote{Formal models whose structure sheaf satisfies this isomorphism are called \textit{integrally closed} in \cite{Dine2025FormalModels}.}.
        \item the connectedness of $\cX$ is equivalent to that of $\cX\genf$ as well as to that of $\cX^\circ$.
    \end{enumerate}
\end{lemma}
\begin{proof}
    (1): Since $\sp_{\cX}$ is glued from open subspaces, it suffices to show that for every smooth connected affine admissible formal scheme $\Spf(A)\to \SKp$, $A$ is integrally closed in $A[\varpi^{-1}]$, i.e., in $A([\varpi^{-1}])^\circ$.  Note that by \cite[Prop.\ 3.4.1]{Luetkebohmert2016RigidGeometryCurves} we have $A([\varpi^{-1}])^\circ=A\otimes_{K^+}K^\circ$, as the special fiber is smooth and hence reduced. Moreover, we have 
    \begin{equation*}
        K^{\circ\circ}(A\otimes_{K^+}K^\circ)=K^{\circ\circ}(K^\circ A)= K^{\circ\circ}A.
    \end{equation*}
    Therefore it suffices to show  that $A/K^{\circ\circ}A$ is integrally closed in $(A\otimes_{K^+}K^\circ)/K^{\circ\circ}(A\otimes_{K^+}K^\circ)$. Geometrically, this is the map
    \begin{equation*}
        \Spf(A)\times_{\SKp} \Spec(k) \to \Spf(A)\times_{\SKp} \Spec(K^+/{K^{\circ\circ}}),
    \end{equation*}
    i.e., the inclusion of the generic fiber of the reduction. Since the reduction is smooth, hence normal and flat, this shows the claim. Indeed, normality of $\Spec(A/K^{\circ\circ}A)$ implies that $A/K^{\circ\circ}A$ is integrally closed in its total ring of fractions $Q$. Flatness implies that $A\otimes_{K^+}k$ embeds into  $Q$. 

    (2): Since $\mathrm{sp}_\cX\colon {\cX\genf}\to\cX$ is continuous and surjective, the connectedness of $\cX\genf$ implies that of $\cX$. Suppose that $\cX$ is connected, and let ${e\in \Gamma(\cX\genf, \cO_{\cX\genf})}$ be idempotent, i.e., a root of $T^2-T=0$. By (1), ${\Gamma(\cX, \cO_{\cX})= \Gamma(\cX\genf, \cO_{\cX\genf}^+)}$  is integrally closed in $\smash{\Gamma(\cX\genf, \cO_{\cX\genf})}$. Hence $\smash{e\in \Gamma(\cX, \cO_{\cX})}$ and therefore $e=0,1$ since $\smash{\cX}$ is connected. This shows the first equivalence. For the second, note that by \cref{l:dense_cocont} $\cX\genf$ is connected if and only if $(\cX\genf)^\circ$ is connected. Hence the statement follows from the previously shown equivalence in the special case $K^+=K^\circ$.
\end{proof}

Note that for every admissible formal scheme $\cX\to \SKp$ the natural map $\cX(\SKp)\to \cX\genf(\SKKp)$ is an isomorphism. This follows from considering affine pieces.

Similarly, a morphism between admissible formal schemes is quasi-compact (quasi-separated, respectively) if and only if its adic generic fiber has the respective property. This is a consequence of the fact that the specialization morphism is surjective. Moreover, observe that every admissible formal scheme over $\SKc$ is quasi-separated, as its reduction is a scheme over a field that is locally of finite type.
We will freely use these facts.

\begin{definition}\label{d:markedNeron_K+}
    Let $G$ be a commutative adic group space over $\SKKp$ that is quasi-separated and locally of topologically finite type. A marked Néron model for $G$ consists of an open subgroup $G'\subseteq G$ and a smooth, separated, fiberwise connected formal model $\cG'\to \SKp$ of $G'$ satisfying the following universal property: 
    
    For every smooth, fiberwise connected, quasi-separated formal scheme $\cX\to \SKp$ with fixed section $x_0\in \cX(\SKp)$, the following natural maps are isomorphisms
    (where pointed means that the fixed section is mapped to the unit section)\vspace{-1mm}
    % https://q.uiver.app/#q=WzAsMyxbMCwwLCJcXG1hdGhybXtNb3J9XntcXG1hdGhybXtwdH19KFxcY1gsIFxcY0cnKSJdLFsxLDAsIlxcbWF0aHJte01vcn1ee1xcbWF0aHJte3B0fX0oXFxjWFxcZ2VuZiwgRycpIl0sWzIsMCwiXFxtYXRocm17TW9yfV57XFxtYXRocm17cHR9fShcXGNYXFxnZW5mLCBHKSJdLFswLDEsIigtKVxcZ2VuZiJdLFsxLDJdXQ==
\[\begin{tikzcd}
	{\mathrm{Mor}^{\mathrm{pt}}(\cX, \cG')} & {\mathrm{Mor}^{\mathrm{pt}}(\cX\genf, G')} & {\mathrm{Mor}^{\mathrm{pt}}(\cX\genf, G).}
	\arrow["{(-)\genf}", from=1-1, to=1-2]
	\arrow[from=1-2, to=1-3]
\end{tikzcd}\]

    A strong marked Néron model of $G$ is defined similarly, requiring that $\cG'$ is smooth, separated, and connected (instead of fiberwise connected), and with the universal property imposed for every smooth, connected formal scheme $\cX\to \SKp$ equipped with a section.
\end{definition}

In other words,
if $X\to \SKKp$ has a $\SKKp$-rational point and admits a smooth formal model with the appropriate connectedness and quasi-separatedness condition, then every morphism $X\to G$ factors as a formal map through $\cG'$ up to a translation on $G$. If $K^+=K^\circ$, \cref{l:smooth_int_closed} shows that the two notions of marked Néron models coincide with that of a marked Néron model in \cite[Rem.\ 7.6.8]{Luetkebohmert2016RigidGeometryCurves}. Note however that in the terminology of loc.\ cit.\ the open subgroup $G'$ alone is the marked N\'eron model.

\begin{remark}\label{r:marked_model_analytic}
    Observe the following implications in the above setting:
    \begin{enumerate}[leftmargin=3em]
    	\item A (strong) marked Néron model is unique if it exists, and the universal property of $\cG'$ equips it with a group structure induced by $G$. If $\cX$ in the previous definition is a group space, then the universal property shows $\mathrm{Hom}(\cX,\cG')\cong \mathrm{Hom}(\cX\genf,G')\cong \mathrm{Hom}(\cX\genf,G)$ for group homomorphisms.
        \item The open subgroup of $G$ determined by its marked (strong marked, respectively) Néron model is the maximal open neighborhood of $e_G$ admitting a smooth formal fiberwise connected (smooth connected, respectively) model over $\SKp$.
        \item If a strong marked N\'eron model $\cG'$ of $G$ exists, then a marked N\'eron model exists and is given by the fiberwise identity component of $\cG'$, i.e., the union of the fiberwise topological connected components of the unit sections. It defines an open quasi-compact subgroup, as this can be checked on the reduction where it is known by the theory of group schemes.
    \end{enumerate}
\end{remark}

The following proposition is a generalization of \cite[Prop.\ 5.4.6]{Luetkebohmert2016RigidGeometryCurves} which deals with the case $K^+=K^\circ$.

\begin{proposition}\label{p:semi-abelian_subgroup}
    Let $G\to \SKKp$ be a connected and commutative adic group space which is separated and locally of topologically finite type.
    Suppose that $G$ contains an open subgroup $G'\subseteq{} G$ that admits a formal semi-abelian group model $\cG'\to \SKp$. Then $\cG'$ is the marked N\'eron model of $G$.
\end{proposition}
\begin{proof}
    Let $X\to \SKKp$ be an adic space admitting a smooth, quasi-separated, fiberwise connected formal model $\cX\to \SKp$ and a fixed section $x_0\in \cX(\SKp)$. Consider the following canonical maps
    % https://q.uiver.app/#q=WzAsNixbMCwwLCJcXG1hdGhybXtNb3J9XntcXG1hdGhybXtwdH19KFxcY1gsIFxcY0cnKSJdLFsxLDAsIlxcbWF0aHJte01vcn1ee1xcbWF0aHJte3B0fX0oWCwgRycpIl0sWzIsMCwiXFxtYXRocm17TW9yfV57XFxtYXRocm17cHR9fShYLCBHKSJdLFswLDEsIlxcbWF0aHJte01vcn1ee1xcbWF0aHJte3B0fX0oXFxjWF5cXGNpcmMsIFxcY0cnXlxcY2lyYykiXSxbMSwxLCJcXG1hdGhybXtNb3J9XntcXG1hdGhybXtwdH19KFheXFxjaXJjLCBHJ15cXGNpcmMpIl0sWzIsMSwiXFxtYXRocm17TW9yfV57XFxtYXRocm17cHR9fShYXlxcY2lyYywgR15cXGNpcmMpIl0sWzAsM10sWzMsNF0sWzQsNV0sWzAsMV0sWzEsMl0sWzEsNF0sWzIsNV1d
\[\begin{tikzcd}[row sep=small]
	{\mathrm{Mor}^{\mathrm{pt}}(\cX, \cG')} & {\mathrm{Mor}^{\mathrm{pt}}(X, G')} & {\mathrm{Mor}^{\mathrm{pt}}(X, G)} \\
	{\mathrm{Mor}^{\mathrm{pt}}(\cX^\circ, \cG'^\circ)} & {\mathrm{Mor}^{\mathrm{pt}}(X^\circ, G'^\circ)} & {\mathrm{Mor}^{\mathrm{pt}}(X^\circ, G^\circ).}
	\arrow[from=1-1, to=1-2]
	\arrow[from=1-1, to=2-1]
	\arrow[from=1-2, to=1-3]
	\arrow[from=1-2, to=2-2]
	\arrow[from=1-3, to=2-3]
	\arrow[from=2-1, to=2-2]
	\arrow[from=2-2, to=2-3]
\end{tikzcd}\]
The lower row consists of isomorphisms by \cite[Prop.\ 5.4.6]{Luetkebohmert2016RigidGeometryCurves}. The right vertical map is injective by \cref{p:valuative_strong_criterion}. The top right horizontal map is clearly injective.
Moreover, the left vertical map is a bijection. Indeed, by \cref{p:pushout_equivalence} we have
\begin{equation*}
    \mathrm{Mor}^{\mathrm{pt}}(\cX, \cG')\cong \mathrm{Mor}^{\mathrm{pt}}(\cX^\circ, \cG'^\circ)\times_{\mathrm{Mor}^{\mathrm{pt}}(\cX_k, \cG'_k)} \mathrm{Mor}^{\mathrm{pt}}(\cX_{K^+/K^{\circ\circ}}, \cG'_{K^+/K^{\circ\circ}}).
\end{equation*}
The projection onto the first factor is a bijection by \cref{p:semi-abelianAreNeron}. Hence all maps in the upper row are bijections.
\end{proof}

\begin{lemma}\label{l:Ext_formal_agrees}
Let $\cB\to \SKp$ be a formal semi-abelian scheme whose $\Spf(K^\circ)$-fiber is a formal abelian scheme. There exists an abeloid variety $B\to \SKKp$ that contains $\cB\genf$ as an open subgroup, and $B$ is unique up to unique isomorphism over $\cB\genf$.
Let $\cT\to \SKp$ be a split formal torus, and denote the corresponding algebraic torus by $T\to \SKKp$. Then we have the following commutative diagram:
    % https://q.uiver.app/#q=WzAsNSxbMCwwLCJcXG1hdGhybXtFeHR9XjFfe1xcbWF0aHJte1phcn19KFxcY0IsXFxjVCkiXSxbMSwwLCJcXG1hdGhybXtFeHR9XjFfe1xcbWF0aHJte2FufX0oXFxjQlxcZ2VuZixUKSJdLFswLDEsIlxcbWF0aHJte0V4dH1eMV97XFxtYXRocm17WmFyfX0oXFxjQl5cXGNpcmMsXFxjVF5cXGNpcmMpIl0sWzEsMSwiXFxtYXRocm17RXh0fV4xX3tcXG1hdGhybXthbn19KEJeXFxjaXJjLFReXFxjaXJjKSJdLFsyLDAsIlxcbWF0aHJte0V4dH1eMV97XFxtYXRocm17YW59fShCLFQpIl0sWzAsMiwiIiwwLHsic3R5bGUiOnsidGFpbCI6eyJuYW1lIjoiaG9vayIsInNpZGUiOiJ0b3AifX19XSxbMSwzLCJcXHNpbSJdLFs0LDEsIlxcc2ltIiwyXSxbMCwxLCIiLDAseyJzdHlsZSI6eyJ0YWlsIjp7Im5hbWUiOiJob29rIiwic2lkZSI6InRvcCJ9fX1dLFsyLDMsIlxcc2ltIl0sWzQsMywiXFxzaW0iXV0=
\[\begin{tikzcd}[row sep=scriptsize]
	{\mathrm{Ext}^1_{\mathrm{Zar}}(\cB,\cT)} & {\mathrm{Ext}^1_{\mathrm{an}}(\cB\genf,T)} & {\mathrm{Ext}^1_{\mathrm{an}}(B,T)} \\
	{\mathrm{Ext}^1_{\mathrm{Zar}}(\cB^\circ,\cT^\circ)} & {\mathrm{Ext}^1_{\mathrm{an}}(B^\circ,T^\circ).}
	\arrow[hook, from=1-1, to=1-2]
	\arrow[hook, from=1-1, to=2-1]
	\arrow["\sim", from=1-2, to=2-2]
	\arrow["\sim"', from=1-3, to=1-2]
	\arrow["\sim", from=1-3, to=2-2]
	\arrow["\sim", from=2-1, to=2-2]
\end{tikzcd}\]
    Moreover, for every $[\cE^\circ]\in \mathrm{Ext}^1_{\mathrm{Zar}}(\cB^\circ,\cT^\circ)$ there exists a finite separable affinoid field extension $(K', K'^+)$ such that $[\cE^\circ]_{(K', K'^\circ)}$ lies in the image of $\smash{\mathrm{Ext}^1_{\mathrm{Zar}}(\cB_{(K', K'^+)},\cT_{(K', K'^+)})}$.
    
    Furthermore, all maps are isomorphisms if $k$ is separably closed.
\end{lemma}
\begin{proof}
    Let $B$ be the unique abeloid variety extending $({\cB\genf})^\circ$ to $\SKKp$, see \cref{l:abeloidEquivalence}. By \cref{p:valuative_strong_criterion,l:circ_reflects_tft_smooth} there is a canonical homomorphism $\cB\genf\to B$ that is an open immersion. The uniqueness claim follows from the properness of $({\cB\genf})^\circ$ and \cref{p:valuative_strong_criterion}.

    We now turn to the comparison of the extension groups appearing in the diagram. The proof of \cite[Thm.\ 3.1]{RibeiroRosengarten2025Extensions} implies that there are natural isomorphisms
    \begin{align*}
        \mathrm{Ext}^1_\mathrm{Zar}(\cB^\circ, \dG_{m,\SKc})\cong\  & H^1_m(\cB^\circ,\dG_m)=\cB^{\circ,\vee}(\SKc)\\
        \mathrm{Ext}^1_\mathrm{an}(B^\circ, \dG_{m,\SKKc})\cong\  & H^1_m(B^\circ,\dG_m)=B^{\circ,\vee}(\SKKc).
    \end{align*}
    Thus
    there is the following commutative square where the vertical maps are isomorphisms
    % https://q.uiver.app/#q=WzAsNCxbMCwwLCJcXG1hdGhybXtFeHR9XjFfe1xcbWF0aHJte1phcn19KFxcY0JeXFxjaXJjLCBcXGNUXlxcY2lyYykiXSxbMCwxLCJcXG1hdGhybXtIb219KFxcY2hpKFxcY1ReXFxjaXJjKSxcXGNCXntcXGNpcmMsXFx2ZWV9KSJdLFsxLDAsIlxcbWF0aHJte0V4dH1eMV97XFxtYXRocm17YW59fShCXlxcY2lyYywgVF5cXGNpcmMpIl0sWzEsMSwiXFxtYXRocm17SG9tfShcXGNoaShUXlxcY2lyYyksQl57XFxjaXJjLFxcdmVlfSkuIl0sWzAsMl0sWzAsMV0sWzIsM10sWzEsMywiKC0pXFxnZW5mIl1d
\[\begin{tikzcd}[row sep=scriptsize]
	{\mathrm{Ext}^1_{\mathrm{Zar}}(\cB^\circ, \cT^\circ)} & {\mathrm{Ext}^1_{\mathrm{an}}(B^\circ, T^\circ)} \\
	{\mathrm{Hom}(\chi(\cT^\circ),\cB^{\circ,\vee})} & {\mathrm{Hom}(\chi(T^\circ),B^{\circ,\vee}).}
	\arrow[from=1-1, to=1-2]
	\arrow[from=1-1, to=2-1]
	\arrow[from=1-2, to=2-2]
	\arrow["{(-)\genf}", from=2-1, to=2-2]
\end{tikzcd}\]
    Note that $(\cB^{\circ,\vee})\genf\cong B^{\circ,\vee}$, see, e.g., \cite[Thm.\ 6.1.1]{Luetkebohmert2016RigidGeometryCurves}. Hence the lower map in the square and in the statement is an isomorphism. The right vertical map in the statement is injective. 
    Indeed, if an extension $T\to X\to \cB\genf$ is given over $\SKKp$, then a retract $X^\circ\to T^\circ$ lifts to a retract $X\to T$ by \cref{p:valuative_strong_criterion}. Similarly  the left vertical map in the statement is injective, where one instead uses \cref{p:pushout_equivalence,p:semi-abelianAreNeron} to lift the retract. The composition $\mathrm{Ext}_{\mathrm{an}}^1(B,T)\to \mathrm{Ext}_{\mathrm{an}}^1(B^\circ,T^\circ)$ is an isomorphism by \cref{l:Ext_abeloid_circ_iso}.

    It remains to show the last two statements of the lemma. 
    To this end, consider the homomorphism defined by pullback
    \begin{equation}\label{eq:proof_ext}\tag{$*$}
        \mathrm{Ext}^1_{\mathrm{Zar}}(\cB,\cT)\longrightarrow \mathrm{Ext}^1_{\mathrm{Zar}}(\cB^\circ,\cT^\circ)\times_{\mathrm{Ext}^1_{\mathrm{fppf}}(\cB_k,\cT_k)}\mathrm{Ext}^1_{\mathrm{fppf}}(\cB_{K^+/K^{\circ\circ}},\cT_{K^+/K^{\circ\circ}}).
    \end{equation}
    We claim that it is an isomorphism. For injectivity suppose that an extension $0\to \cT\to \cG\to \cB\to 0$ admits a retract homomorphism  $r^\circ\colon\cG^\circ\to \cT^\circ$ on the $\SKc$-fiber. Since $\cT_{K^+/K^{\circ\circ}}$ is the marked Néron model of $\cT_k$ by \cref{p:semi-abelianAreNeron}, $r_k$ lifts to a retract $\cG_{K^+/K^{\circ\circ}}\to\cT_{K^+/K^{\circ\circ}}$. By \cref{p:pushout_equivalence}, this gives rise to a retract $\cG\to \cT$. For surjectivity note that an element on the right-hand side determines a sequence of homomorphisms $\cT\to \cH\to \cB$ over $\SKp$ by \cref{p:pushout_equivalence}. Then $\cT\to \cH$ is seen to be the kernel and $\cH\to \cB$ is smooth. We need to show that it is Zariski locally split. Let $x\in \cB$ be given. By Hilbert 90, the reduction $\cH_{K^+/K^{\circ\circ}}\to \cB_{K^+/K^{\circ\circ}}$ is Zariski locally split. If $x\in \cU\subseteq \cB$ is an affine neighborhood which admits a section to $\cH$ on the reduction, then smoothness of $\cH\times_{\cB}\cU\to \cU$ implies that the section lifts to a section $\cU\to \cH\times_{\cB}\cU$. Indeed, one can lift the section to a section over $\Spec(K^+/{\varpi K^\circ)}$ for a suitable choice of $\varpi\in K^{\circ\circ}\setminus \{0\}$ using \spcite{01ZL}, and then use that $\cU$ is affine and formal smoothness to lift this section along thickenings $\Spec(K^+/{(\varpi K^\circ)^n)}\to \Spec(K^+/{(\varpi K^\circ)^{n+1})}$. This shows
    that \cref{eq:proof_ext} is an isomorphism.
   
    Since finite separable residue extensions lift, 
    the last two statements of the lemma are reduced to the analogous statements for $\mathrm{Ext}^1(\cB_{K^+/K^{\circ\circ}},\cT_{K^+/K^{\circ\circ}})\to \mathrm{Ext}^1(\cB_k,\cT_k)$. This is an isomorphism when $k$ is separably closed by \cref{l:semi-abelian_ext}. The remaining statement follows
   by approximating a separable closure of $k$ by finite subextensions.
\end{proof}

We now combine the preceding results to construct Néron models of abeloid varieties over $\SKKp$.

\begin{theorem}\label{t:uniformization_formal1}
    Let $A$ be an abeloid variety over $\SKKp$.
    After a finite separable affinoid field extension, $A$ 
    admits a marked Néron model $\mathcal{E}$ that fits into a short exact sequence
        % https://q.uiver.app/#q=WzAsNSxbMSwwLCJcXGNUIl0sWzIsMCwiXFxvdmVybGluZXtcXHJlc3Ryc21hbGx7RX17XFxjQlxcZ2VuZn19Il0sWzMsMCwiXFxjQiJdLFswLDAsIjAiXSxbNCwwLCIwIl0sWzMsMF0sWzAsMV0sWzEsMl0sWzIsNF1d
        \[\begin{tikzcd}
        	0 & \cT & {\cE} & \cB & 0
        	\arrow[from=1-1, to=1-2]
        	\arrow[from=1-2, to=1-3]
        	\arrow[from=1-3, to=1-4]
	        \arrow[from=1-4, to=1-5]
        \end{tikzcd}\]
        where $\cT$ is a split formal torus and $\cB$ is a formal semi-abelian scheme whose $\SKc$-fiber is a formal abelian scheme.
        
        If $k$ is algebraically closed or $K^+/K^{\circ\circ}$ is a discrete valuation ring, then, after such a base change, $A$ admits a strong marked Néron model, and $\cE$ is its fiberwise identity component.
\end{theorem}
\begin{proof}
    Let $K'/K$ be a finite separable field extension such that the conclusions of \cref{t:classical_uniformization_formal1} hold for $A^\circ\times_{\SKKc} \Spa(K', K'^\circ)$. By \spcite{00IA} there exists a valuation ring $V'$ of the residue field of $K'^\circ$ dominating $K^+/{K^{\circ\circ}}$. Then  the preimage of $V'$ under the reduction map $K'^\circ\to K'^\circ/K'^{\circ\circ}$
    is a valuation ring of $K'$ dominating $K^+$. It is open as it contains $K'^{\circ\circ}$, and therefore defines an affinoid field $S$.
    We replace $A$ by the base change of $A$ to $S$, and $\SKKp$ by $S$.

    Let $0\to \cT^\circ\to\cE^\circ\to \cB^\circ\to 0$ be the extension of the marked Néron model of $A^\circ\to \SKKc$ as in \cref{t:classical_uniformization_formal1}. We may further choose $S$ such that the reduction $\cB^\circ\times_{\SKc}\Spec(k)$ extends to a semi-abelian scheme $\cB_{K^+/K^{\circ\circ}}$ over $\Spec(K^+/K^{\circ\circ})$, see \cref{l:semi-abelian-model}. By \cref{p:pushout_equivalence} this induces an extension of $\cB^\circ$ to a formal semi-abelian scheme $\cB$ over $\SKp$. After replacing $S$ by a further finite separable extension, the extension $\cE^\circ$ lifts to an extension $\cE$ of $\cB$ by a lift of the split formal torus $\cT$, see \cref{l:Ext_formal_agrees}. By \cref{p:valuative_strong_criterion,l:circ_reflects_tft_smooth}, the open immersion $(\cE^\circ)\genf\to A$ extends to an open immersion $\cE\genf\to A$ containing $e_A$. It  is the marked Néron model of $A$ by \cref{p:semi-abelian_subgroup}.

    The existence of a strong marked Néron model follows in the same way: We again start from $\cE^\circ$ and extend it to $\SKp$ using \cref{p:pushout_equivalence} and the existence of algebraic Néron models on the reduction, i.e., \cref{t:Neron_existence_V} and \cref{e:Neron_excistence}(1) in the algebraically closed case and \cite[Cor.\ 1.3.2, Prop.\ 10.1.7]{BoschLuetkebohmertRaynaud1990NeronModels} in the discretely valued case.
    Then we argue exactly as in the proof of \cref{p:semi-abelian_subgroup} using \cref{l:smooth_int_closed}(2) that it has the claimed universal property.
\end{proof}

We note that the proof concerning strong marked Néron models in the above theorem only uses the existence of Néron models over $\Spec(K^+/K^{\circ\circ})$ for semi-abelian varieties over $\Spec(k)$. Thus, if $K^+/K^{\circ\circ}$ satisfies the condition in \cref{t:Neron_existence_V}, then one still obtains the existence of strong marked Néron models after possibly a further finite affinoid base field extension.

\subsection{Relation to Raynaud uniformization}
In analogy to the classical situation of an abeloid variety $A^\circ\to \SKKc$ where its Raynaud uniformization naturally arises from the pushout of the generic fiber of the marked Néron model of $A^\circ$ (cf.\ \cref{t:classical_uniformization_formal1}), we can now show the following.

\begin{corollary}\label{c:uniformization_formal2}
    Let $A$ be an abeloid variety over $\SKKp$. 
    Assume that $(K,K^+)$ is big enough so that $A$ admits a marked N\'eron model $\mathcal{E}$ which is an extension of a formal semi-abelian scheme $\cB$ with formal abelian $\SKc$-fiber by a split formal torus $\cT$.
    \begin{enumerate}[leftmargin=3em]
        \item There exists an abeloid variety $\smash{B\to \SKKp}$ that contains $\smash{\cB\genf}$ as an open subgroup, and it is unique up to unique isomorphism over $\smash{\cB\genf}$. Let $T$ be the full algebraic torus containing $\cT\genf$. Then the pushout of the generic fiber 
        uniquely extends to an extension $T\to E\to B$, i.e., the lower right square in the following diagram is cartesian
        % https://q.uiver.app/#q=WzAsMTUsWzEsMCwiXFxjVFxcZ2VuZiJdLFszLDAsIlxcY0JcXGdlbmYiXSxbMSwxLCJUIl0sWzIsMSwiRSciXSxbMywxLCJcXGNCXFxnZW5mIl0sWzAsMCwiMCJdLFswLDEsIjAiXSxbNCwwLCIwIl0sWzQsMSwiMCJdLFsyLDAsIlxcY0VcXGdlbmYiXSxbMCwyLCIwIl0sWzQsMiwiMCJdLFsxLDIsIlQiXSxbMiwyLCJFJ157Y30iXSxbMywyLCJCIl0sWzAsMl0sWzIsM10sWzMsNF0sWzEsNCwiIiwwLHsibGV2ZWwiOjIsInN0eWxlIjp7ImhlYWQiOnsibmFtZSI6Im5vbmUifX19XSxbMywwLCIiLDAseyJzdHlsZSI6eyJuYW1lIjoiY29ybmVyIn19XSxbNSwwXSxbNiwyXSxbMSw3XSxbNCw4XSxbMCw5XSxbOSwxXSxbOSwzXSxbMiwxMiwiIiwwLHsibGV2ZWwiOjIsInN0eWxlIjp7ImhlYWQiOnsibmFtZSI6Im5vbmUifX19XSxbMTIsMTNdLFsxMywxNF0sWzMsMTNdLFs0LDE0XSxbMywxNCwiIiwwLHsic3R5bGUiOnsibmFtZSI6ImNvcm5lciJ9fV0sWzEwLDEyXSxbMTQsMTFdXQ==
\[\begin{tikzcd}[row sep=0.9em]
	0 & {\cT\genf} & {\cE\genf} & {\cB\genf} & 0 \\
	0 & T & {E'} & {\cB\genf} & 0 \\
	0 & T & {E} & B & 0.
	\arrow[from=1-1, to=1-2]
	\arrow[from=1-2, to=1-3]
	\arrow[from=1-2, to=2-2]
	\arrow[from=1-3, to=1-4]
	\arrow[from=1-3, to=2-3]
	\arrow[from=1-4, to=1-5]
	\arrow[equals, nfold, from=1-4, to=2-4]
	\arrow[from=2-1, to=2-2]
	\arrow[from=2-2, to=2-3]
	\arrow[equals, nfold, from=2-2, to=3-2]
	\arrow["\lrcorner"{anchor=center, pos=0.125, rotate=180}, draw=none, from=2-3, to=1-2]
	\arrow[from=2-3, to=2-4]
	\arrow[from=2-3, to=3-3]
	\arrow["\lrcorner"{anchor=center, pos=0.125}, draw=none, from=2-3, to=3-4]
	\arrow[from=2-4, to=2-5]
	\arrow[from=2-4, to=3-4]
	\arrow[from=3-1, to=3-2]
	\arrow[from=3-2, to=3-3]
	\arrow[from=3-3, to=3-4]
	\arrow[from=3-4, to=3-5]
\end{tikzcd}\]
    \item The open immersion $\cE\genf\subseteq A$ uniquely extends as a homomorphism $\cE\genf\subseteq E\to A$, and the latter map is the Raynaud uniformization of $A$ (in the sense of \cref{t:uniformization_analytic}).
    \end{enumerate}
\end{corollary}
\begin{proof}
    The first sentence
    of (1) is a part of \cref{l:Ext_formal_agrees}. The existence and uniqueness of the extension $T\to E\to B$ is the upper right horizontal isomorphism in \cref{l:Ext_formal_agrees}.

     Note that $A$ admits a Raynaud uniformization
     by \cref{t:classical_uniformization_formal1}(2) and \cref{t:uniformization_analytic}(3). After taking the $\SKKc$-fiber, the lower row of the diagram is the Raynaud extension of $A^\circ$ since $(\smash{\cB\genf})^\circ = B^\circ$.
     As passage to the $\SKKc$-fiber induces bijections on morphisms into $E$ and $A$,
     (2) follows from the analogous classical statement.
\end{proof}

The preceding description also identifies precisely when the classical rank one picture carries over to $\SKKp$.

\begin{corollary}\label{c:uniformization_formal3}
    Let $A$ be an abeloid variety over $\SKKp$. Assume that {$(K,K^+)$} is big enough so that $A$ admits a marked N\'eron model $\mathcal{E}$ which is an extension of a formal semi-abelian scheme $\cB$ with formal abelian $\SKc$-fiber by a split formal torus $\cT$. Then the following statements are equivalent:
    \begin{enumerate}
        \item The generically good reduction part $B$ in the Raynaud uniformization of $A$ has good reduction.
        \item The Raynaud extension of $A$ arises as the pushout of the generic fiber of its marked Néron model.
        \item The fiberwise torus rank of the marked Néron model is constant (i.e., $\cB$ is a formal abelian scheme).
    \end{enumerate}
\end{corollary}
\begin{proof}
    Adopt the notation in \cref{c:uniformization_formal2}.
    The statements in (3) are equivalent by \cite[Cor.\ I.2.11]{FaltingsChai1990Degeneration} which states that a semi-abelian scheme with constant torus rank zero is an abelian scheme.

    Suppose that (1) holds. Denote the formal abelian group model of $B$ by $\cB'$. As $\cB$ is formal semi-abelian, it is the marked Néron model of $B$ by \cref{p:semi-abelian_subgroup}. Hence $\cB'\genf\to B$ factors through $\cB\genf$, showing that $\cB$ is a model of $B$ and hence proper by \cref{l:generic_fiber_properties}. Thus (1) implies (3).
    Suppose that (3) holds. Then $\cB\genf\to B$ is proper and hence surjective, showing (2).
    Suppose that (2) holds. Then $\cB$ is a group model of $B$, showing (1).
\end{proof}

\subsection{Further properties of N\'eron models}
We begin by showing a more involved mapping property of strong marked N\'eron models.

\begin{lemma}\label{l:strong_Neron_mapping}
	Let $A\to \SKKp$ be an abeloid variety with strong marked Néron model $\cN\to \SKp$. Suppose that $k$ is algebraically closed, and that the fiberwise identity component of $\cN$ is a formal semi-abelian scheme. 
    Let $\cX\to \SKp$ be a smooth connected formal scheme. Then all maps in the following diagram are bijections
	% https://q.uiver.app/#q=WzAsNixbMCwwLCJcXE1vcihcXGNYLCBcXGNOKSJdLFsxLDAsIlxcTW9yKFxcY1hcXGdlbmYsIFxcY05cXGdlbmYpIl0sWzIsMCwiXFx7ZlxcaW4gXFxNb3IoXFxjWFxcZ2VuZiwgQSlcXG1pZCBmKFxcY1hcXGdlbmYpXFxjYXBcXGNOXFxnZW5mXFxuZXEgXFxlbXB0eXNldFxcfSJdLFswLDEsIlxcTW9yKFxcY1heXFxjaXJjLCBcXGNOXlxcY2lyYykiXSxbMSwxLCJcXE1vcigoXFxjWF5cXGNpcmMpXFxnZW5mLCAoXFxjTl5cXGNpcmMpXFxnZW5mKSJdLFsyLDEsIlxce2ZcXGluIFxcTW9yKChcXGNYXlxcY2lyYylcXGdlbmYsIEFeXFxjaXJjKVxcbWlkIGYoKFxcY1heXFxjaXJjKVxcZ2VuZilcXGNhcChcXGNOXlxcY2lyYylcXGdlbmZcXG5lcSBcXGVtcHR5c2V0XFx9Il0sWzAsM10sWzAsMV0sWzMsNF0sWzEsMl0sWzQsNV0sWzEsNF0sWzIsNV1d
\[\begin{tikzcd}[column sep=0.83em,row sep=small]
	{\Mor(\cX, \cN)} & {\Mor(\cX\genf, \cN\genf)} & {\{f\hspace{-0.2mm}\in\hspace{-0.2mm} \Mor(\cX\genf, A)\mid f(\cX\genf)\cap\cN\genf\neq \emptyset\}} \\
	{\Mor(\cX^\circ, \cN^\circ)} & {\Mor((\cX^\circ)\genf, (\cN^\circ)\genf)} & {\{f\hspace{-0.2mm}\in\hspace{-0.2mm} \Mor((\cX^\circ)\genf, A^\circ)\mid f((\cX^\circ)\genf)\cap(\cN^\circ)\genf\neq \emptyset\}.}
	\arrow[from=1-1, to=1-2]
	\arrow[from=1-1, to=2-1]
	\arrow[from=1-2, to=1-3]
	\arrow[from=1-2, to=2-2]
	\arrow[from=1-3, to=2-3]
	\arrow[from=2-1, to=2-2]
	\arrow[from=2-2, to=2-3]
\end{tikzcd}\]
\end{lemma}
\begin{proof}
	Starting from $\cN^\circ$ it follows as in the last paragraph of the proof of \cref{t:uniformization_formal1} that $A$ admits a strong marked Néron model extending $\cN^\circ$ and whose base change to $K^+/K^{\circ\circ}$ is the Néron model of $\cN_k$. Thus the uniqueness of the strong marked Néron model shows that $\cN_{K^+/K^{\circ\circ}}$ is the Néron model of $\cN_k$. Hence the left vertical map is an isomorphism by \cref{p:pushout_equivalence}. Since $k$ is algebraically closed, $\cX^\circ$ admits a $\SKc$-rational point. Thus the lower horizontal map in the left square is an isomorphism. The right vertical map is well defined, as the intersection contains a height one point if it is nonempty. The middle vertical and right vertical maps are injective by \cref{p:valuative_strong_criterion}.

    We claim that if $(K', K'^+)$ is a finite affinoid base field extension, then $\cN_{\Spf(K'^+)}$ is the strong marked Néron model of $A_{\Spa(K', K'^+)}$. Indeed, its fiberwise identity component is formal semi-abelian. Thus it suffices to show that the reduction is the Néron model of $\cN_{K'^\circ/K'^{\circ\circ}}$. As $k$ is algebraically closed, the residue extension of valued fields $(K'^\circ/K'^{\circ\circ})/k$ is trivial. Hence $\cN_{K'^+/K'^{\circ\circ}}=\cN_{K^+/K^{\circ\circ}}$ along this identification, which implies the claim.

    We claim that the lower right map is an isomorphism. It is clearly injective. Let $f$ be in the lower right set. 
    To show that it factors through $(\cN^\circ)\genf$, we may first pass to a finite field extension of $K$ by the previous paragraph (cf.\ \cite[Cor.\ 4.2.5]{Conrad2006RelativeAmpleness}
    for the descent statements). Therefore, we may assume that $f^{-1}(f((\cX^\circ)\genf)\cap(\cN^\circ)\genf)=(\cX^\circ)\genf\times_{f,A^\circ}(\cN^\circ)\genf$ contains a $\SKKc$-rational point. The statement then follows from the strong marked Néron property by considering translations.
\end{proof}

\begin{lemma}\label{l:algebraic_vs_formal_Neron}
    Suppose that $A\to \Spec(K)$ is an abelian variety that admits a semi-abelian model $\cE\to \Spec(K^+)$. There is a canonical homomorphism
    \begin{equation*}
        (\overhat{\cE})\genf\to A^{\mathrm{an}},
    \end{equation*}
    where $\overhat{\cE}$ is the formal completion of $\cE$ along $V(\varpi)$,
    and $A^{\mathrm{an}}$ is the analytification to $\SKKp$. This makes $\overhat{\cE}$ the marked Néron model of $A^{\mathrm{an}}$.
\end{lemma}
\begin{proof}
     By \cite[Thm.\ 5.3.1]{Conrad1999IrreducibleComponents} there is a natural homomorphism $((\overhat{\cE})^\circ)\genf\to A^{\mathrm{an},\circ}$, and it is an open immersion. \Cref{p:valuative_strong_criterion,l:circ_reflects_tft_smooth} imply that it uniquely extends to an open immersion $\smash{(\overhat{\cE})}\genf\to A^{\mathrm{an}}$. Since $\smash{\overhat{\cE}}\times_{\SKp}\Spec(K^+/{K^{\circ\circ}})=\cE\times_{\Spec(K^+)}\Spec(K^+/{K^{\circ\circ}})$, it is a formal semi-abelian scheme. The statement follows from \cref{p:semi-abelian_subgroup}.
\end{proof}

We now investigate finiteness of the multiplication by $m$ map $[m]$ on strong marked Néron models and algebraic Néron models.

\begin{proposition}\label{p:Neron_finite_m}
	\begin{enumerate}[leftmargin=2.6em]
		\item Let
        $A\to \SKKp$ be an abeloid variety with strong marked Néron model $\cN\to \SKp$. Suppose that $k$ is algebraically closed, and that the fiberwise identity component of $\cN$ is semi-abelian.
        Then the maps $[m]_{\cN}$ and $[m]_{\cN\genf}$ are finite, flat and surjective for every nonzero $m$.
		\item Suppose that $V$ is a valuation ring with algebraically closed field of fractions $F$. Let
        $E_F\to \Spec(F)$ be a semi-abelian variety with Néron model $N\to \Spec(V)$. Then $[m]_{N}$ is finite, flat and surjective for every nonzero $m$.
	\end{enumerate}
\end{proposition}
    
\begin{proof}
	Let us first show (1) under the assumption that $K$ is algebraically closed. Note that $[m]_A$ is finite, flat and surjective. Indeed, it is proper as a map between proper spaces. The properties can therefore be checked on the $\SKKc$-fiber by \cref{l:circ_reflect_pp,l:circ_reflects_tft_smooth}. {In the rank one case} the statements are known.
        
    Let $\cE$ denote the fiberwise identity component of $\cN$. Then $[m]_{\cE_{K^+/K^{\circ\circ}}}$ is quasi-finite, flat, and surjective (cf.\ \cite[Lemmas 7.3.1, 7.3.2]{BoschLuetkebohmertRaynaud1990NeronModels}). Let $s\in \Spec(K^+/K^{\circ\circ})$, and let $C$ be a connected component of $\cN_s$. Since $\kappa(s)$ is algebraically closed there exists some $c\in C(s)$. {Considering the corresponding translation to the identity component, we deduce 
that $[m]_{\cN_{K^+/K^{\circ\circ}}}$ is flat and quasi-finite on $C$. Hence it is flat and}
locally quasi-finite by the fiberwise flatness criterion. Moreover, surjectivity follows from the divisibility of fiberwise component groups, which was shown in 
\cref{l:divisible_component_group}. Flatness of the reduction implies that $[m]_\cN$ is flat, and in particular that $\cN[m]$ is an admissible formal scheme.
    
    Next, we claim that $(\cN\genf)[m]=A[m]\cap \cN\genf$ is finite. Every connected component of $A[m]^\circ$ is topologically a point that admits a section as $K$ is algebraically closed. Partial properness of $A[m]$ implies that these extend to surjective sections over $\SKKp$. Hence $\cN\genf$ fully contains every connected component it intersects since $\cN\genf(\SKKc)=\cN\genf(\SKKp)$ by \cref{l:strong_Neron_mapping}. Thus $(\cN\genf)[m]$ is finite,
    and hence $\cN[m]$ is proper as properness descends from the adic generic fiber by \cref{l:generic_fiber_properties}.

    In particular, local quasi-finiteness of the reduction implies that $\cN_{K^+/K^{\circ\circ}}[m]$ is finite. Therefore $[m]_{\cN_{K^+/K^{\circ\circ}}}$ is finite by descent, which implies that $[m]_{\cN}$ is finite. The statement for $[m]_{\cN\genf}$ follows.

    We claim that this implies (2). Indeed, consider $K:=F((t))^{\mathrm{alg},\wedge}$ with the $t$-adic topology. Its residue field is $F$, and $K$ therefore admits a microbial valuation ring $K^+$ such that $K^+/K^{\circ\circ}=V$.
    
    Let $E_{K^\circ}:=E_F\times_{\Spec(F)}\Spec(K^\circ)$ be the pullback of $E_F$ along the section $F\to F[[t]]\to K^\circ$ of the reduction map $K^\circ\to F$. Then $E_{K^\circ}\to \Spec(K^\circ)$ is globally an extension of an abelian scheme by a split torus. Its special fiber over the residue field $F$ is $E_F$.
    
    The completion $\widehat{E_{K^\circ}}\to \SKc$ is an extension of a formal abelian scheme by a split formal torus. Consider the pushout $E$ of its adic generic fiber along the full algebraic torus. Choose any full lattice $M\subseteq E$. Then the quotient $E/M$ is an abeloid variety over $\SKKc$ (cf.\ \cite[Prop.\ 6.1.4]{Luetkebohmert2016RigidGeometryCurves}). It follows that $\widehat{E_{K^\circ}}$ is the marked Néron model of $E/M$.  Denote by $A\to \SKKp$ the unique extension of $E/M$ to an abeloid variety over $\SKKp$ due to \cref{p:equivalence_proper}. By \cref{t:uniformization_formal1}, $\widehat{E_{K^\circ}}$ extends to a strong marked Néron model $\cN\to \SKp$ of $A$, and its reduction $\cN_V$ is the Néron model of $\cN_F=(\widehat{E_{K^\circ}})_F=E_F$. Thus $\cN_V=N$, and hence (2) follows from (1) applied to $A$.

    This in turn implies (1) for algebraically closed residue fields. Indeed, the properties can be checked on the reduction $\cN_{K^+/K^{\circ\circ}}$.
    But $\cN_{K^+/K^{\circ\circ}}$ is the Néron model of $\cN_k$ as seen in the beginning of the proof of \cref{l:strong_Neron_mapping}. Thus (2) applies to $\cN_{K^+/K^{\circ\circ}}$.
\end{proof}

Observe that the preceding proof of (2) reduces an algebraic statement about a valuation ring to an analytic statement about another valuation ring of strictly higher rank. We then apply our uniformization theory to the latter analytic setup.

Many statements from the classical theory over $\SKKc$ readily carry over to $\SKKp$. For example, let us record how morphisms between abeloid varieties can be described in terms of their Raynaud uniformizations and marked Néron models.

\begin{proposition}\label{p:Raynauduni2}
    Let $A_1, A_2\to \SKKp$ be abeloid varieties and {assume that {$(K,K^+)$} is big enough so that $A_i$ admit marked N\'eron models $\cE_i$ which are extensions of formal semi-abelian schemes $\cB_i$ with formal abelian $\SKc$-fiber by split formal tori $\cT_i$.} Denote the Raynaud uniformizations by $A_i\cong E_i/M_i$, $B_i\cong E_i/T_i$.
    Then
    \begin{enumerate}
        \item $\cE_i$ is the marked Néron model of $E_i$ for $i=1,2$.
        \item There is a canonical isomorphism
        \begin{equation*}
            \Hom(A_1, A_2)\cong \{f\in \Hom(E_1, E_2)\mid f(M_1)\subseteq M_2\}.
        \end{equation*}
        \item Let $\varphi_{E_i}\colon \underline{\chi(T_i)}\to B_i^{\vee}$ be the homomorphisms corresponding to $E_i$ along the isomorphisms $\mathrm{Hom}(\underline{\chi(T_i)}, B_i^\vee) \cong {\mathrm{Ext}}^1_{\mathrm{an}}(B_i, T_i)$. There is a natural isomorphism
        \[
	    \Hom(E_1, E_2) \cong\left\{\hspace{0.5mm}f\in \Hom(\cE_1, \cE_2)\hspace{2mm}\middle|
    \hspace{1.7mm}\vcenter{\hbox{%
	        \parbox{0.20\linewidth}{\centering\noindent
			The following
            diagram commutes:
	        }
			}}
    % https://q.uiver.app/#q=WzAsNCxbMCwwLCJcXHVuZGVybGluZXtcXGNoaShUXzIpfSJdLFswLDEsIkJfMl5cXHZlZSJdLFsxLDAsIlxcdW5kZXJsaW5le1xcY2hpKFRfMSl9Il0sWzEsMSwiQl8xXlxcdmVlIl0sWzAsMSwiXFx2YXJwaGlfe0VfMn0iXSxbMCwyXSxbMSwzXSxbMiwzLCJcXHZhcnBoaV97RV8xfSJdXQ==
\begin{tikzcd}[cramped,row sep=small]
	{\underline{\chi(T_2)}} & {\underline{\chi(T_1)}} \\
	{B_2^\vee} & {B_1^\vee}
	\arrow[from=1-1, to=1-2]
	\arrow["{\varphi_{E_2}}", from=1-1, to=2-1]
	\arrow["{\varphi_{E_1}}", from=1-2, to=2-2]
	\arrow[from=2-1, to=2-2]
\end{tikzcd}
            \right\},
            \]
            where the top horizontal arrow is induced by the restriction of $f$ to the tori, and the bottom arrow is the dual of the morphism $B_1\to B_2$ induced by the morphism $\cB_1\to \cB_2$ on quotients.
    \end{enumerate}
\end{proposition}
\begin{proof}
(1): Consider the diagram in \cref{c:uniformization_formal2}(1). Then $(\cE_i)\genf\to E_i'$ is an open immersion as $\cT_i\genf\to T_i$ is an open immersion, and $E_i'\to E_i$ is an open immersion by pullback. Thus (1) follows from \cref{p:semi-abelian_subgroup}.
(2) follows from the naturality of the uniformization cross (\cref{t:uniformization_analytic}(1)) and $A_i\cong E_i/M_i$.
    
(3): Observe that every morphism $(\cB_1)\genf\to(\cB_2)\genf\to B_2$ indeed uniquely extends to a morphism $B_1\to B_2$ by \cref{p:valuative_strong_criterion}. Moreover, note that every homomorphism $(\cE_1)\genf\to E_1\to E_2$ induces
a unique homomorphism $\cE_1\to \cE_2$ by (1). Thus the claim follows from \cite[Prop.\ 6.1.9]{Luetkebohmert2016RigidGeometryCurves}. That statement shows that in the classical case $K^+=K^\circ$, $\Hom(E_1,E_2)$ is given by the set of commutative squares on the right, where the horizontal maps may be arbitrary homomorphisms. 
\end{proof}

\bigskip

To illustrate our results, let us consider abeloid varieties of relative dimension $1$, i.e., those that are analytifications of elliptic curves. Recall that one-dimensional abeloid varieties over $\SKKc$ are projective (cf.\ \cite[Thm.\ 1.8.1]{Luetkebohmert2016RigidGeometryCurves}). By \cite[Cor.\ 1.8.7]{Huber1996EtaleCohomology}, $(-)^\circ$ reflects and preserves the property of having pure relative dimension $d$ for morphisms over $\SKKp$ that are locally of topologically finite type. Thus abeloid varieties over $\SKKp$ of relative dimension $1$ indeed correspond to elliptic curves over $\Spec(K)$.
The following example is inspired by \cite[Example 2.20]{Heuermodular}.

\begin{example}\label{e:classification_RU_curves}
Let $A_{\mathrm{alg}}$ be an elliptic curve over an algebraically closed non-Archimedean field $K$ with analytification $A:=A_{\mathrm{alg}}^{\mathrm{an}}$ over $\SKKp$.
 By \cref{t:Neron_existence_V}, $A_\mathrm{alg}$ admits a Néron model $N\to \Spec(K^+)$  and its identity component $\cE_{\mathrm{alg}}\to \Spec(K^+)$ is semi-abelian. The completion $\cE:=\smash{\overhat{\cE_{\mathrm{alg}}}}$ is naturally the marked Néron model of $A$ by \cref{l:algebraic_vs_formal_Neron}.
    \begin{enumerate}[label=(\arabic*)]
        \item If $j(A_{\mathrm{alg}})\in K^+$,  then $A$ has good reduction over $\SKKp$. The torus and lattice in the Raynaud uniformization are $0$.
        \item If $j(A_{\mathrm{alg}})\in K^\circ\setminus K^+$, then $A$ has generically good reduction but does not have good reduction. Hence the torus and lattice in the Raynaud uniformization are again $0$. The $\SKc$-fiber of $\cE$ is proper, but its fibers eventually become toric in the direction of the valuation ring.
        \item If $j(A_{\mathrm{alg}})\notin K^\circ$, then $A$ has totally degenerate reduction, i.e., $\cE\cong \dG_{m,\SKp}$ and $B=0$. The open immersion $\cE\genf\to A$ uniquely extends to a homomorphism $\dG_{m,\SKKp}\to A$. It yields the Raynaud uniformization $A\cong \dG_{m,\SKKp}/\langle \hspace{0.15mm}q \hspace{0.15mm}\rangle$ for some $0\neq q\in K^{\circ\circ}$.
    \end{enumerate}
    For (2) and (3), in contrast to the classical case $K^+=K^\circ$, the subgroup defined by $\cE\genf\to A$ is generally not the largest connected smooth-formal open neighborhood of $e_A$. Instead, this role is taken by the adic generic fiber of the strong marked Néron model $\cN$ of $A$, which exists by \cref{t:uniformization_formal1}. Its fiberwise identity component is $\cE$, and it is glued from copies of $\cE$. Note that $\cN$ is generally not the completion of $N$.

   Moreover,
   in case (3), for $s\in \SKp$ the component groups are
    \begin{equation*}
        \pi_0(N_s)\cong \Gamma_{K^+_s}/\lvert q\rvert_{K^+_s}^{\dZ},\quad\pi_0(\cN_s)\cong \Gamma_{K^+_s/K^{\circ\circ}}.
    \end{equation*}
\end{example}
\begin{proof}
    First, note that if $V$ is a valuation ring with algebraically closed fraction field $F$, then an elliptic curve $E \to \Spec(F)$ extends to an elliptic curve over $\Spec(V)$ if and only if $j(E)\in V$. This follows from, e.g., the proofs of \cite[Prop.\ VII.5.5, Cor.\ A.1.4b]{Silverman2009EllipticCurves}, which do not need any assumption on the valuation ring. In particular, if $p\in \Spec(K^+/K^{\circ\circ})$,
    then $(\cE_{\mathrm{alg}})_p=\cE_p$ is an abelian variety if and only if $j(A_{\mathrm{alg}})\in K^+_p$. Otherwise it is a torus.

    Hence in case (1) $A$ has good reduction over $\SKKp$.
    The torus rank in the Raynaud uniformization is the torus rank of $\cE_{k}$, which shows (2). In case (3) the fiberwise torus rank of $\cE$ is $1$. Hence $\cE\cong\dG_{m,\SKp}$, and the claim follows from \cref{c:uniformization_formal2}.

    In cases (2) and (3) it follows from the strong marked Néron mapping property that $\cN\genf$  is the largest connected smooth-formal neighborhood of $e_A$. 
    For $K^+=K^\circ$, the two notions of marked Néron models agree.

For the component group calculation in case (3), we may base change $K^+$ to $K^+_{s}$. We therefore abuse notation and assume that $s$ is the closed point of $\SKp$. Note that since $k$ is algebraically closed, every connected component of $N_s$ (of $\cN_s$, respectively) is hit by a section of $\Spec(K^+)$ (of $\Spec(K^+/K^{\circ\circ})$, respectively). Thus 
    \begin{equation*}
        \pi_0(N_s)\cong N(K^+)/\cE_{\mathrm{alg}}(K^+),\quad \pi_0(\cN_s)=\cN(K^+/K^{\circ\circ})/\cE(K^+/K^{\circ\circ}).
    \end{equation*}

    Let us first consider $N_s$. By \cref{c:uniformization_formal2} the following diagram commutes, and the horizontal arrows are isomorphisms
    % https://q.uiver.app/#q=WzAsOCxbMCwxLCJOKEteKykiXSxbMCwwLCJcXGNFX3tcXG1hdGhybXthbGd9fShLXispIl0sWzEsMSwiQV9cXG1hdGhybXthbGd9KEspIl0sWzEsMCwiXFxjRShLXispIl0sWzIsMCwiXFxjRVxcZ2VuZihLLCBLXispIl0sWzIsMSwiQShLLCBLXispIl0sWzMsMSwiS15cXHRpbWVzL3FeXFxkWiJdLFszLDAsIktee1xcY2lyYyxcXHRpbWVzfSJdLFswLDJdLFsyLDVdLFs1LDZdLFsxLDNdLFszLDRdLFs0LDddLFsxLDBdLFs3LDZdLFs0LDVdXQ==
\[\begin{tikzcd}[row sep=small]
	{\cE_{\mathrm{alg}}(K^+)} & {\cE(K^+)} & {\cE\genf(K, K^+)} & {K^{+,\times}} \\
	{N(K^+)} & {A_\mathrm{alg}(K)} & {A(K, K^+)} & {K^\times/q^\dZ.}
	\arrow[from=1-1, to=1-2]
	\arrow[from=1-1, to=2-1]
	\arrow[from=1-2, to=1-3]
	\arrow[from=1-3, to=1-4]
	\arrow[from=1-3, to=2-3]
	\arrow[from=1-4, to=2-4]
	\arrow[from=2-1, to=2-2]
	\arrow[from=2-2, to=2-3]
	\arrow[from=2-3, to=2-4]
\end{tikzcd}\]
    Therefore, $\pi_0(N_s)\cong (K^\times/q^\dZ)/ K^{+,\times}\cong \Gamma_{K^+}/\lvert q\rvert_{K^+}^{\dZ}$ along the valuation map.

    For $\cN_s$ we have $\cN(K^+/K^{\circ\circ})=\cN(k)=k^\times$, and $\cE(K^+/K^{\circ\circ})=(K^+/{K^{\circ\circ}})^\times$. Hence the valuation map induces an isomorphism $\pi_0(\cN_s)\cong \Gamma_{K^+/K^{\circ\circ}}$.
\end{proof}

\vspace{5ex}
	\noindent	
	Simon Kaib \hfill Annette Werner \\
 Institut f\"ur Mathematik  \hfill Institut f\"ur Mathematik \\
	Goethe-Universit\"at Frankfurt \hfill Goethe-Universit\"at Frankfurt\\
	Robert-Mayer-Str. 6-8 \hfill Robert-Mayer-Str. 6-8\\
	60325 Frankfurt am Main \hfill 60325 Frankfurt am Main\\	
	kaib@math.uni-frankfurt.de \hfill werner@math.uni-frankfurt.de
	
\end{document}